\documentclass[letterpaper, 11pt,  reqno]{amsart}

\usepackage{amsmath,amssymb,amscd,amsthm,amsxtra, esint}
\usepackage{tikz} 

\usepackage{comment}
\usepackage{mathtools}
\usepackage{color}
\usepackage[implicit=true]{hyperref}
\usepackage{scalerel} 

\usepackage{cases}

\usepackage[margin=1.2in,marginparwidth=1.5cm, marginparsep=0.5cm]{geometry}

\allowdisplaybreaks[2]

\usepackage{color}

\definecolor{gr}{rgb}   {0.,   0.69,   0.23 }
\definecolor{bl}{rgb}   {0.,   0.5,   1. }
\definecolor{mg}{rgb}   {0.85,  0.,    0.85}
\definecolor{yl}{rgb}   {0.8,  0.7,   0.}
\definecolor{or}{rgb}  {0.7,0.2,0.2}

\usetikzlibrary{shapes.misc}
\usetikzlibrary{shapes.symbols}
\usetikzlibrary{shapes.geometric}

\tikzset{
	ddot/.style={circle,fill=white,draw=black,inner sep=0pt,minimum size=0.8mm},
	>=stealth,
	}

\tikzset{
	ddot2/.style={circle,fill=black,draw=black,inner sep=0pt,minimum size=0.8mm},
	>=stealth,
	}

\newtheorem{theorem}{Theorem} [section]

\newtheorem{lemma}[theorem]{Lemma}
\newtheorem{proposition}[theorem]{Proposition}
\newtheorem{remark}[theorem]{Remark}

\newcommand{\I}{\mathcal{I}}

\newcommand{\noi}{\noindent}
\newcommand{\Z}{\mathbb{Z}}
\newcommand{\R}{\mathbb{R}}
\newcommand{\C}{\mathbb{C}}
\newcommand{\T}{\mathbb{T}}

\newcommand{\E}{\mathbb{E}}

\newcommand{\F}{\mathcal{F}}

\newcommand{\al}{\alpha}
\newcommand{\be}{\beta}
\newcommand{\dl}{\delta}

\newcommand{\nb}{\nabla}

\newcommand{\Dl}{\Delta}
\newcommand{\eps}{\varepsilon}

\newcommand{\g}{\gamma}
\newcommand{\G}{\Gamma}
\newcommand{\ld}{\lambda}

\newcommand{\s}{\sigma}

\newcommand{\ft}{\widehat}

\newcommand{\wt}{\widetilde}
\newcommand{\cj}{\overline}

\newcommand{\dt}{\partial_t}

\newcommand{\embeds}{\hookrightarrow}

\newcommand{\ta}{\theta}

\renewcommand{\l}{\ell}
\renewcommand{\o}{\omega}
\renewcommand{\O}{\Omega}

\newcommand{\les}{\lesssim}

\newcommand{\jb}[1]
{\langle #1 \rangle}

\newcommand{\ind}{\mathbf 1}

\newcommand{\NN}{\mathcal{N}}

\newcommand{\N}{\mathbb{N}}

\renewcommand{\H}{\mathcal{H}}

\newcommand{\HS}{\textit{HS}}

\newtheorem*{ackno}{Acknowledgements}

\newcommand{\too}{\longrightarrow}

\newcommand{\Id}{\textup{Id}}

\numberwithin{equation}{section}
\numberwithin{theorem}{section}

\makeatletter
\@namedef{subjclassname@2020}{%
  \textup{2020} Mathematics Subject Classification}
\makeatother

\begin{document}
\baselineskip = 13.5pt

\title[GWP of the energy-critical SNLW]
{Global well-posedness of the energy-critical  \\ stochastic nonlinear wave equations}

\author[E.~Brun, G.~Li, and R.~Liu]
{Enguerrand Brun, Guopeng Li, and Ruoyuan Liu}

\address{
Enguerrand Brun,  Ecole Normale Sup\'erieure de Lyon\\
46, all\'ee d'Italie\\
69364-Lyon Cedex 07\\
France
}
\email{enguerrand.brun@ens-lyon.fr}

\address{
Guopeng Li, School of Mathematics\\
The University of Edinburgh\\
and The Maxwell Institute for the Mathematical Sciences\\
James Clerk Maxwell Building\\
The King's Buildings\\
Peter Guthrie Tait Road\\
Edinburgh\\ 
EH9 3FD\\
 United Kingdom
}
\email{guopeng.li@ed.ac.uk}

\address{
Ruoyuan Liu,  School of Mathematics\\
The University of Edinburgh\\
and The Maxwell Institute for the Mathematical Sciences\\
James Clerk Maxwell Building\\
The King's Buildings\\
Peter Guthrie Tait Road\\
Edinburgh\\ 
EH9 3FD\\
 United Kingdom}

\email{ruoyuan.liu@ed.ac.uk}


\subjclass[date]{35L71, 35R60, 60H15}
%


\begin{abstract}
We consider the Cauchy problem for the defocusing energy-critical stochastic nonlinear wave equations (SNLW) with an additive stochastic forcing on $\R^{d}$ and $\T^{d}$ with $d \geq 3$.
By adapting the probabilistic perturbation argument 
employed 
in the context of the random data Cauchy theory
 by B\'enyi-Oh-Pocovnicu (2015) and Pocovnicu
(2017)  and in the context of stochastic PDEs by Oh-Okamoto (2020),  
we prove global well-posedness of the defocusing energy-critical SNLW.
In particular, on $\T^d$, we prove global well-posedness with the stochastic forcing below the energy space.

\end{abstract}

\maketitle

\tableofcontents

\section{Introduction}
\label{SEC:intro}

We consider the following Cauchy problem for the defocusing energy-critical stochastic nonlinear wave equation (SNLW) on $\mathcal{M} = \R^{d}$ or $\mathbb{T}^{d}$ (with $\T = \R / 2 \pi \Z$) for $d \geq 3$:
\begin{equation}
\begin{cases}
\dt^{2} u - \Dl u + |u|^{\frac{4}{d - 2}}u = \phi \xi \\
(u,\dt u)|_{t = 0} = (u_{0},u_{1}),
\end{cases}
\label{SNLW}
\end{equation}
where $u$ is real-valued, $\xi$ is the space-time white noise on $\R_{+}\times \mathcal{M}$, and $\phi$ is a bounded operator on $L^{2} (\mathcal{M})$. 
The aim of this paper is to show global well-posedness of \eqref{SNLW}.


Let us first mention some backgrounds on the energy-critical NLW. Consider the following deterministic defocusing NLW on $\R^d$ with $d \geq 3$:
\begin{align} 
\dt^{2} u - \Dl u + |u|^{\frac{4}{d - 2}}u=0.
\label{NLWec}
\end{align}

\noi
It is well known that the following dilation symmetry for $\ld > 0$
\begin{align*}
u (t, x) \mapsto u_\ld (t, x) := \ld^{\frac{d - 2}{2}} u (\ld t, \ld x)
\end{align*}

\noi
maps solutions of NLW \eqref{NLWec} to solutions of NLW \eqref{NLWec}. A direct computation yields
\begin{align*}
\| u \|_{\dot{H}^1 (\R^d)} = \| u_\ld \|_{\dot{H}^1 (\R^d)},
\end{align*}

\noi
so that the scaling critical Sobolev regularity for NLW \eqref{NLWec} is $s_c = 1$.
Also, the energy defined by
\begin{align}
E(\vec u) = E (u, \dt u) := \frac 12 \int (\dt u)^{2} dx + \frac 12 \int |\nabla u|^{2} dx + \frac{d - 2}{2d} \int |u|^{\frac{2d}{d-2}} dx
\label{energy}
\end{align}

\noi
is conserved under the flow of \eqref{NLWec}. In view of the Sobolev embedding $\dot{H}^1 (\R^d) \embeds L^{\frac{2d}{d - 2}} (\R^d)$, we see that $E (\vec u) < \infty$ if and only if $\vec u \in \dot{\H}^1 (\R^d) := \dot{H}^1 (\R^d) \times L^2 (\R^d)$. For this reason, we refer to $\dot{\H}^1 (\R^d)$ as the energy space for NLW \eqref{NLWec}. Moreover, we say that NLW \eqref{NLWec} is {\it energy-critical}.  On $\T^d$, although there is no dilation symmetry, the heuristics provided by the scaling analysis still hold and we say that \eqref{NLWec} on $\T^d$ is energy-critical.

Note that for energy-subcritical NLW, after proving local well-posedness, one can easily obtain global well-posedness by using the conservation of the energy, which provides an a priori control of the $\dot{\H}^1$-norm of the solution. For the energy-critical NLW, however, there is a delicate balance between the linear and nonlinear parts of the equation. In the energy-critical setting, only the energy conservation itself is not enough to obtain global well-posedness, which makes the problem quite intricate. Still, after substantial efforts of many mathematicians, we now know that the energy-critical defocusing NLW \eqref{NLWec} is globally well-posed in $\dot{\H}^1 (\R^d)$ and all solutions in the energy space scatter. See \cite{Str, Gri90, Gri92, SS93, SS94, Kap, BS, BG, Nak1, Nak2, Tao06}. On the other hand, in the periodic setting, the global well-posedness results for \eqref{NLWec} in $\dot{\H}^1 (\R^d)$ immediately implies corresponding global well-posedness for \eqref{NLWec} in $\H^1 (\T^d)$, thanks to the finite speed of propagation. We also point out that these well-posedness results for \eqref{NLWec} in the energy space are sharp in the sense that ill-posedness for \eqref{NLWec} on $\R^d$ occurs below the energy space; see \cite{CCT, FO, OOTz}.

\medskip
Let us now go back to the defocusing energy-critical NLW with a stochastic forcing \eqref{SNLW}. Well-posedness theory of SNLW has been studied extensively; see \cite{PZ00, Pes02, On04, Chow06, On07, DKMNX, DS, On10-1, On10-2, BOS16, BR22} and the references therein. More recently, there has been a significant development in the well-posedness theory of singular SNLW with an additive stochastic forcing; see \cite{GKO, GKO2, GKOT, OWZ, ORTz, OOT1, OOT2, OTWZ}.

Our main goal in this paper is to prove global well-posedness of the defocusing energy-critical SNLW \eqref{SNLW}. We say that $u$ is a solution to \eqref{SNLW} if it satisfies the following Duhamel formulation (or mild formulation):
\begin{align}
u(t) = V(t) (u_0, u_1) - \int_{0}^{t} S(t-t') \big( |u|^{\frac{4}{d - 2}} u \big) (t')  dt' + \int_0^t S(t - t') \phi \xi (d t'), 
\label{Duh}
\end{align}

\noi
where
\begin{align}
S(t) = \frac{\sin (t |\nabla|)}{|\nabla|} \quad \text{and} \quad V(t) (\phi_0, \phi_1) = \dt S(t) \phi_0 + S(t) \phi_1.
\label{defS}
\end{align}

\noi
The last term on the right-hand-side of \eqref{Duh} is called the stochastic convolution, which we denote by
\begin{align}
\Psi (t) = \int_0^t S(t - t') \phi \xi (d t').
\label{Psi}
\end{align}

\noi
See \eqref{Psi2} in Subsection \ref{SUB:LEM_Rd} for a precise definition.

On $\R^d$, we obtain the following global well-posedness result.

\begin{theorem}
\label{THM:R}
Let $d \geq 3$ and $\phi \in \HS (L^{2} (\R^d),L^{2} (\R^d))$. Then, the defocusing energy-critical SNLW \eqref{SNLW} is globally well-posed in $\dot{\H}^1 (\R^d)$ in the sense that the following statement holds true almost surely; given $(u_0, u_1) \in \dot{\H}^1 (\R^d)$, there exists a global-in-time solution $u$ to \eqref{SNLW} with $(u, \dt u)|_{t = 0} = (u_0, u_1)$.
\end{theorem}

Here, the assumption on $\phi$ is chosen in such a way that the stochastic convolution $\Psi$ lies in the energy space $\dot{\H}^1 (\R^d)$. See Lemma \ref{LEM:Psi} below.

For the proof of Theorem \ref{THM:R}, we first establish local well-posedness of \eqref{SNLW} in Section \ref{SEC:LWP} and then prove global well-posedness of \eqref{SNLW} in Subsection \ref{SUBSEC:GWP_R} below. To prove well-posedness of \eqref{SNLW}, we use the first order expansion $u = \Psi + v$ and consider the following equation for $v$:
\begin{align}
\begin{cases}
\dt^2 v - \Dl v + |v + \Psi|^{\frac{4}{d - 2}} (v + \Psi) = 0 \\
(v, \dt v)|_{t = 0} = (u_0, u_1).
\end{cases}
\label{SNLWv1}
\end{align}

\noi
By viewing \eqref{SNLWv1} as an energy-critical NLW for $v$ with a perturbation term, we adapt the {\it probabilistic perturbation theory} developed in \cite{BOP, Poc45} in the context of random data Cauchy theory. See Lemma \ref{LEM:pert} below. The perturbation theory has also been previously used to prove global well-posedness for other equations in deterministic settings, stochastic settings, or random initial data settings. See \cite{TVZ, CKSTT08, KM, KOPV, BOP, OOP, OO}. 

In order to apply the perturbation argument for the equation \eqref{SNLWv1}, we need to make sure that the stochastic convolution $\Psi$ is small on short time intervals. This can be done since $\Psi$ can be bounded in Strichartz spaces (i.e.~$L^q$ in time and $L^r$ in space), as in \cite{OO}. Nevertheless, due to the complicated nature of the wave equations, establishing the space-time regularity of $\Psi$ is non-trivial. See Lemma \ref{LEM:Psi} for more details.

Another important ingredient in carrying out the perturbation argument in our setting is an {\it a priori energy bound}. This is achieved by a rigorously justified application of Ito's lemma. See also \cite{dBD, OO, CL, CLO}.

\medskip
We now switch our attention to the periodic setting. Here, we assume that $\phi$ is a a Fourier multiplier operator on $L^2 (\T^d)$. Namely, for any $n \in \Z^d$,
\begin{align*}
\phi e^{in \cdot x} = \ft \phi_n e^{in \cdot x}
\end{align*}

\noi
for some $\ft \phi_n \in \C$.

\begin{theorem}
\label{THM:T}
Let $d \geq 3$, $\phi$ be a Fourier multiplier operator on $L^2 (\T^d)$, and $\phi \in \HS (L^{2} (\T^d),H^s (\T^d))$ with $s \in \R$ satisfying
\begin{align*}
\textup{(i)}~ d = 3: ~s > -\frac 12; \qquad \textup{(ii)}~ d = 4: ~ s > -1; \qquad \textup{(iii)}~ d \geq 5: ~s \geq -1.
\end{align*}
Then, the defocusing energy-critical SNLW \eqref{SNLW} is globally well-posed in $\H^1 (\T^d)$ in the sense that the following statement holds true almost surely; given $(u_0, u_1) \in \H^1 (\T^d)$, there exists a global-in-time solution $u$ to \eqref{SNLW} with $(u, \dt u)|_{t = 0} = (u_0, u_1)$.
\end{theorem}

Note that compared to the $\R^d$ case, there is an improvement in the regularity assumption of the noise term. This is mainly because of the better space-time regularity of the stochastic convolution $\Psi$ on a bounded domain. See Lemma \ref{LEM:Phi} for more details. Moreover, our global well-posedness result for $d \geq 5$ is optimal in the sense that the regularity of $\phi$ cannot be further lowered. See Remark \ref{RMK:opt} below.

Thanks to the finite speed of propagation, the proof of Theorem \ref{THM:T} follows from the same strategies for the $\R^d$ case. However, due to the lower regularity assumption of the noise, the stochastic convolution does not belong to the energy space $\H^1 (\T^d)$. Thus, for the a priori energy bound, instead of using Ito's lemma, we use a Gronwall-type argument developed by \cite{BTz14, OP16} in the context of random data Cauchy theory.

\medskip
We conclude this introduction by stating several remarks.
\begin{remark} \rm
\text{(i)} In \cite{OP16, Poc45}, the authors studied the defocusing energy-critical NLW \eqref{NLWec} on $\R^d$ with initial data below the energy space. In particular, using the Wiener randomization of the initial data, they proved global well-posedness of \eqref{NLWec} by establishing an energy bound via a Gronwall-type argument. 

However, at this point, we do not know whether one can prove any global well-posedness results for the defocusing energy-critial SNLW \eqref{SNLW} with initial data below the energy space. The main obstacle is that, even with randomized initial data, the Gronwall-type argument as in \cite{BTz14, OP16, Poc45} is not directly applicable due the lack of space-time regularity of the stochastic convolution $\Psi$.

\smallskip \noi
\text{(ii)} Compared to the $\R^d$ case, the situation is better on $\T^d$ and we can prove almost sure global well-posedness of the stochastic energy-critical defocusing NLW \eqref{SNLW} below the energy space. Specifically, we consider the following equation for $d \geq 3$:
\begin{align}
\begin{cases}
\dt^2 u - \Dl u + |u|^{\frac{4}{d - 2}} u = \phi \xi \\
(u, \dt u)|_{t = 0} = (u_0^\o, u_1^\o),
\end{cases}
\label{rand_init}
\end{align}

\noi
where $\phi$ satisfies the same condition as in Theorem \ref{THM:T} and $(u_0^\o, u_1^\o)$ is a randomization of $(u_0, u_1)$ defined by
\begin{align*}
(u_0^\o, u_1^\o) := \bigg( \sum_{n \in \Z^d} g_{n, 0} (\o) \ft{u_0} (n) e^{in \cdot x}, \sum_{n \in \Z^d} g_{n, 1} (\o) \ft{u_1} (n) e^{in \cdot x} \Bigg),
\end{align*}

\noi
where $\{g_{n, j}\}_{n \in \Z^d, j \in \{0, 1\}}$ is a sequence of independent mean zero complex-valued random variables conditioned such that $g_{-n, j} = \cj{g_{n, j}}$ for all $n \in \Z^d$. Moreover, we assume that there exists a constant $c > 0$ such that on the probability distributions $\mu_{n, j}$ of $g_{n, j}$, we have
\begin{align*}
\int e^{\g \cdot x} d \mu_{n, j} (x) \leq e^{c |\g|^2}, \quad j = 0, 1
\end{align*}

\noi
for all $\g \in \R^2$ when $n \in \Z^d \setminus \{0\}$ and all $\g \in \R$ when $n = 0$. 

Then, due to better integrability of $V(t) (u_0^\o, u_1^\o)$ compared to $V(t) (u_0, u_1)$ for any $t \geq 0$ (see, for example, \cite[Proposition 4.1 and Proposition 4.4]{OP17}), one can show the following result. Given $(u_0, u_1) \in \H^{s} (\T^d)$ with $s \in \R$ satisfying
\begin{align*}
\textup{(i)}~ d = 3: ~s > \frac 12; \qquad \textup{(ii)}~ d = 4: ~ s > 0; \qquad \textup{(iii)}~ d \geq 5: ~s \geq 0,
\end{align*}
there exists a global-in-time solution $u$ to the equation \eqref{rand_init} almost surely.
\end{remark}

\begin{remark} \rm
\label{RMK:opt}
As stated in Theorem \ref{THM:T}, we obtain global well-posedness of \eqref{SNLW} with $\phi \in \HS (L^2 (\T^d), H^s (\T^d))$, where $s > -1$ if $d = 4$ and $s \geq -1$ if $d \geq 5$. Note that when $s < -1$, the stochastic convolution $\Psi$ is merely a distribution, and hence a proper renormalization is needed to make sense of the power-type nonlinearity.  For this purpose, the nonlinearity must be an integer power of the form $u^k$, which is only true for $d \leq 4$. This means that our global well-posedness result of \eqref{SNLW} on $\T^d$ for $d \geq 5$ is the best that we can achieve.

In the case $d = 4$, it is natural to ask whether it is possible to extend the well-posedness theory of \eqref{SNLW} on $\T^4$ with the stochastic convolution $\Psi$ being merely a distribution. The answer is no. Indeed, in a recent preprint \cite{OOPTz}, the authors showed an ill-posedness result for the following (renormalized) stochastic NLW on $\T^d$:
\begin{align*}
\dt^2 u + (1 - \Dl) u + u^k = \phi \xi,
\end{align*}

\noi
where $k \geq 2$ is an integer and $\phi$ is a Fourier multiplier operator with $\phi \in \HS (L^2 (\T^d), H^s (\T^d))$, $s < -1$. Hence, our global well-posedness result for the defocusing energy-critical SNLW \eqref{SNLW} on $\T^4$ is sharp up to the endpoint. It would be of interest to show global well-posedness of \eqref{SNLW} in the endpoint case $s = -1$ for $d = 4$, which seems to require a more intricate analysis. See also Remark \ref{RMK:dtE} (ii) below.
\end{remark}

\section{Preliminary results and lemmas}
In this section, we recall some notations, definitions, useful lemmas, and previous results.

For two positive numbers $A$ and $B$, we use $A \les B$ to denote $A \leq C B$ for some constant $C > 0$. Also, we use shorthand notations for space-time function spaces, such as $C_T H^s_x$ for $C([0, T], H^s (\R^d))$ or $C([0, T], H^s (\T^d))$.

We recall that if $H_{1}$, $H_{2}$ are Hilbert spaces, then for a linear operator $\phi$ from $H_{1}$ to $H_{2}$, we denote
\[ \| \phi \|_{\HS (H_{1},H_{2})} = \Big( \sum_{n \in \N} \| \phi e_{n} \|_{H_{2}}^{2} \Big)^{1 / 2} \]
as the Hilbert-Schmidt operator norm of $\phi$, where $\{e_{n}\}_{n \in \N}$ is an orthonormal basis of $H_{1}$.

\subsection{Preliminary lemmas on Sobolev spaces}
\label{SUB:LEM_Rd}
In this subsection, we recall Sobolev spaces on $\R^d$ and $\T^d$ and also some useful estimates. 

For $s \in \R$, we denote by $\dot{H}^s (\R^d)$ the homogeneous $L^2$-based Sobolev space with the norm
\begin{align*}
\| f \|_{\dot{H}^s (\R^d)} := \big\| |\xi|^s \ft f (\xi) \big\|_{L^2_\xi (\R^d)},
\end{align*}

\noi
where $\ft f$ is the Fourier transform of $f$. We denote by $\dot{H}^s (\R^d)$ the inhomogeneous $L^2$-based Sobolev space with the norm
\begin{align*}
\| f \|_{H^s (\R^d)} := \big\| \jb{\xi}^s \ft f (\xi) \big\|_{L^2_\xi (\R^d)},
\end{align*}   

\noi
where $\jb{\cdot} = (1 + |\cdot|^2)^{\frac 12}$. We also define
\begin{align*}
\dot{\H}^s (\R^d) := \dot{H}^s (\R^d) \times \dot{H}^{s - 1} (\R^d) \quad \text{and} \quad \H^s (\R^d) := H^s (\R^d) \times H^{s - 1} (\R^d).
\end{align*}

\noi
For $1 < p \leq \infty$, we denote by $W^{s, p} (\R^d)$ the $L^p$-based Sobolev space with the norm
\begin{align*}
\| f \|_{W^{s, p} (\R^d)} := \big\| \F^{-1} \big( \jb{\xi}^s \ft f (\xi) \big) \big\|_{L^p (\R^d)},
\end{align*}

\noi
where $\F^{-1}$ denotes the inverse Fourier transform.

On $\T^d$, for $s \in \R$, we denote by $H^s (\T^d)$ the inhomogeneous $L^2$-based Sobolev space with the norm
\begin{align*}
\| f \|_{H^s (\T^d)} &:= \big\| \jb{n}^s \ft f (n) \big\|_{\l_n^2 (\Z^d)}.
\end{align*}

\noi
We also define
\begin{align*}
\H^s (\T^d) := H^s (\T^d) \times H^{s - 1} (\T^d).
\end{align*}


\medskip
We now recall some useful estimates for Sobolev spaces, starting with the following fractional chain rule. For a proof, see \cite{CW}.\footnote{Here, we use the version in \cite[Theorem 3.3.1]{Sta}. As pointed out in \cite{Sta}, the proof in \cite{CW} needs a small correction, which yields the fractional chain rule in a less general context. See also \cite{Kato, Sta, Tay}.}
\begin{lemma}
\label{LEM:chain}
Let $d \geq 1$, $s \in (0,1)$, and $r > 2$. Let $1 < p, p_1 < \infty$ and $1 < p_2 \leq \infty$ satisfying $\frac{1}{p} = \frac{1}{p_1} + \frac{1}{p_2}$. Let $F$ denote the function $F(u) = |u|^{r-1}u$ or $F(u) = |u|^r$. Then, we have
\begin{align*}
\| F(u) \|_{W^{s,p}(\R^d)} \les \| u \|_{W^{s,  p_1}(\R^d)} \big\| |u|^{r-1} \big\|_{L^{p_2}(\R^d)}.
\end{align*}
\end{lemma}

We also need the following Gagliardo-Nirenberg interpolation inequality. The proof of this inequality follows directly from Sobolev's inequality and interpolation.
\begin{lemma}\label{LEM:Gag}
Let $d \geq 1$, $1 < p_1, p_2 < \infty$, and $s_1,s_2 > 0$. Let $p > 1$ and $\ta \in (0,1)$ satisfying
\[ -\frac{s_1}{d} + \frac{1}{p} = (1-\ta)\bigg( - \frac{s_2}{d} + \frac{1}{p_1}  \bigg) + \frac{\ta}{p_2} \quad \text{and} \quad s_1 \leq (1-\ta)s_2. \]
Then, we have
\[ \| u \|_{W^{s_1,p}(\R^d)} \les \| u \|_{W^{s_2,p_1}(\R^d)}^{1-\ta} \| u \|_{L^{p_2}(\R^d)}^\ta. \]
\end{lemma}

\subsection{Previous results on wave equations}
In this subsection, we record some results on wave equations. 

Let us first recall the Strichartz estimate for linear wave equations. Let $\g \in \R$ and $d \geq 2$. We say that a pair $(q, r)$ is $\dot{H}^\g (\R^d)$-wave admissible if $q \geq 2$ and $2 \leq r < \infty$ satisfy
\begin{align*}
\frac 1q + \frac{d - 1}{2r} \leq \frac{d - 1}{4} \quad \text{and} \quad \frac 1q + \frac dr = \frac d2 - \g.
\end{align*}

\noi
The following Strichartz estimates for wave equations are well-studied. See, for example, \cite{GSV, LS, KT}.
\begin{lemma}
\label{LEM:str}
Let $\g > 0$, $d \geq 2$, $(q, r)$ be $\dot{H}^\g (\R^d)$-wave admissible, and $(\wt q, \wt r)$ be $\dot{H}^{1 - \g} (\R^d)$-wave admissible. Let $u$ be the solution of
\begin{align*}
\begin{cases}
\dt^2 u - \Dl u + F = 0 \\
(u, \dt u)|_{t = 0} = (u_0, u_1)
\end{cases}
\end{align*}

\noi
on $I \times \R^d$, where $I \subset \R$ is an interval containing 0. Then, we have
\begin{align*}
\| u \|_{L_I^\infty \dot{H}_x^\g} + \| \dt u \|_{L_I^\infty \dot{H}_x^{\g - 1}} + \| u \|_{L_I^q L_x^r} \les \| u_0 \|_{\dot{H}^\g} + \| u_1 \|_{\dot{H}^{\g - 1}} + \| F \|_{L_I^{\wt q '} L_x^{\wt r '}},
\end{align*}

\noi
where $\wt q '$ and $\wt r '$ are H\"older conjugates of $\wt q$ and $\wt r$, respectively.
\end{lemma}

In the energy-critical setting, we will frequently use the following Strichartz space:
\begin{align}
\| u \|_{X (I)} = \| u \|_{X(I \times \R^d)} := \| u \|_{L_I^{\frac{d + 2}{d - 2}} L_x^{\frac{2(d + 2)}{d - 2}} (\R^d)},
\label{defX}
\end{align}

\noi
where one can easily check that the pair $(\frac{d + 2}{d - 2}, \frac{2(d + 2)}{d - 2})$ is $\dot{H}^1$-wave admissible. Another frequently used pair is $(\infty, 2)$, which is $\dot{H}^0$-wave admissible.

\medskip
We now recall the following global space-time bound on the solution to the deterministic defocusing energy-critical NLW.
\begin{lemma}
\label{LEM:bdd}
Let $d \geq 3$ and $(w_0, w_1) \in \dot{\H}^1 (\R^d)$. Let $w$ be the solution of the energy-critical defocusing NLW:
\begin{align*}
\begin{cases}
\dt^2 w - \Dl w + |w|^{\frac{4}{d - 2}} w = 0 \\
(w, \dt w)|_{t = 0} = (w_0, w_1).
\end{cases}
\end{align*}

\noi
Then, 
\begin{align*}
\| w \|_{X (\R_+ \times \R^d)} < C \big( \| (w_0, w_1) \|_{\dot{\H}^1 (\R^d)} \big),
\end{align*}

\noi
where $C (\cdot) > 0$ is a non-decreasing function.
\end{lemma}

For a proof of Lemma \ref{LEM:bdd}, see \cite[Lemma 4.6]{Poc45} and also the references therein, where the steps can be easily extended to $d \geq 3$.

\medskip
Lastly, we recall the following long-time perturbation lemma from \cite[Lemma 4.5]{Poc45}. The lemma in \cite{Poc45} was stated for $d = 4, 5$, but it can be easily extended to $d \geq 3$.
\begin{lemma}
\label{LEM:pert}
Let $d \geq 3$, $(v_0, v_1) \in \dot{\H}^1 (\R^d)$, and $M > 0$. Let $I \subset \R_+$ be a compact time interval and $t_0 \in \I$. Let $v$ be a solution on $I \times \R^d$ of the following perturbed equation
\begin{align*}
\begin{cases}
\dt^2 v - \Dl v + |v|^{\frac{4}{d - 2}} v = f \\
(v, \dt v)|_{t = t_0} = (v_0, v_1)
\end{cases}
\end{align*}

\noi
with
\begin{align*}
\| v \|_{X (I \times \R^d)} \leq M.
\end{align*}

\noi
Let $(w_0, w_1) \in \dot{\H}^1 (\R^d)$ and let $w$ be the solution of the defocusing energy-critical NLW:
\begin{align*}
\begin{cases}
\dt^2 w - \Dl w + |w|^{\frac{4}{d - 2}} w = 0 \\
(w, \dt w)|_{t = 0} = (w_0, w_1).
\end{cases}
\end{align*}

\noi
Then, there exists $\wt{\eps} (M) > 0$ sufficiently small such that if
\begin{align*}
\| (v_0 - w_0, v_1 - w_1) \|_{\dot{\H}^1} &\leq \eps, \\
\| f \|_{L_I^1 L_x^2} &\leq \eps
\end{align*}

\noi
for some $0 < \eps < \wt{\eps} (M)$, then the following holds:
\begin{align*}
\| (v - w, \dt v - \dt w ) \|_{L_I^\infty \dot{\H}_x^1} + \| v - w \|_{L_I^q L_x^r} \leq C(M) \eps
\end{align*}

\noi
for all $\dot{H}^1$-wave admissible pairs $(q, r)$, where $C(\cdot)$ is a non-decreasing function.
\end{lemma}

\subsection{Regularity of stochastic convolutions}
In this subsection, we study regularity properties of several stochastic object.

We first consider the stochastic convolution $\Psi$ as defined in \eqref{Psi}, which we now provide a more precise definition. By fixing an orthonormal basis $\{ e_k \}_{k \in \N}$ of $L^2 (\R^d)$, we define $W$ as the following cylindrical Wiener process on $L^2 (\R^d)$:
\begin{align*}
W (t) := \sum_{k \in \N} \be_k (t) e_k.
\end{align*}

\noi
Here, $\be_k$ is defined by $\be_k (t) = \langle \xi, \ind_{[0, t]} \cdot e_k \rangle_{t, x}$, where $\langle \cdot, \cdot \rangle_{t, x}$ denotes the duality pairing on $\R_+ \times \R^d$. As a result, $\{ \be_k \}_{k \in \N}$ is a family of mutually independent real-valued Brownian motions. 
The stochastic convolution $\Psi$ is then given by
\begin{align}
\Psi (t) = \int_0^t S(t - t') \phi dW(t') = \sum_{k \in \N} \int_0^t S(t - t') \phi e_k d \beta_k (t'),
\label{Psi2}
\end{align} 

\noi
where $S$ is the Fourier multiplier as defined in \eqref{defS}.

We now show the following lemma, which establishes the regularity property of $\Psi$. 
\begin{lemma}
\label{LEM:Psi}
Let $d \geq 1$ and $T > 0$. Suppose that $\phi \in \HS (L^2 (\R^d), H^s (\R^d))$ for some $s \in \R$.

\smallskip \noi
\textup{(i)} We have $\Psi \in C ([0, T]; H^{s + 1} (\R^d))$ almost surely. Moreover, for any finite $p \geq 1$, we have
\begin{align*}
\E \Big[ \sup_{0 \leq t \leq T} \| \Psi (t) \|_{H^{s + 1} (\R^d)}^p \Big] \leq C \| \phi \|_{\HS (L^2, H^s)}^p
\end{align*} 

\noi
for some constant $C = C(T, p) > 0$.

\smallskip \noi
\textup{(ii)} Let $1 \leq q < \infty$, $d \geq 3$, and $r_d = \frac{2 (d + 2)}{d - 2}$. Let $\g \in (0, 1)$ such that $1 - \g \leq \min (\frac{1}{d + 2}, \frac{6 - d}{2 (d + 2)})$. Then, we have $\Psi \in L^q ([0, T]; W^{s + 1 - \g, r_d} (\R^d))$ almost surely. Moreover, for any finite $p \geq 1$, we have
\begin{align}
\E \Big[ \| \Psi \|_{L_T^q W_x^{s + 1 - \g, r_d} (\R^d)}^p \Big] \leq C \| \phi \|_{\HS (L^2, H^s)}^p
\label{reg_goal}
\end{align}

\noi
for some constant $C = C(T, p) > 0$.
\end{lemma}

\begin{proof}
For part (i), see \cite{DPZ}. For part (ii), the proof is similar to \cite[Lemma 2.1 (ii)]{OPW} with some additional cares. We define $\phi_{\leq K}$ be such that
\begin{align*}
\phi_{\leq K} e_k = 
\begin{cases}
\phi e_k & \text{ if } k \leq K \\
0 & \text{ if } k > K,
\end{cases}
\end{align*}

\noi
and we define $\Psi_{\leq K}$ be as in \eqref{Psi2} with $\phi$ replaced by $\phi_{\leq K}$. Note that $\phi_{\leq K}$ converges to $\phi$ in $\HS (L^2, H^s)$ as $K \to \infty$. We fix $1 \leq q < \infty$ and choose $\wt q = \wt q (d, \g) \geq 2$ such that $(\wt q, r_d)$ is $\dot{H}^\g$-wave admissible, which can be satisfied given $1 - \g \leq \min (\frac{1}{d + 2}, \frac{6 - d}{2 (d + 2)})$. Then, for $p \geq \max(q, r_d)$, by Minkowski's integral inequality, the Gaussian hypercontractivity, and the Ito isometry, we have
\begin{align}
\begin{split}
\big\| \| \Psi_{\leq K} \|_{L_T^q W_x^{s + 1 - \g, r_d}} \big\|_{L^p (\O)} &\leq \big\| \| \jb{\nabla}^{s + 1 - \g} \Psi_{\leq K} \|_{L^p (\O)} \big\|_{L_T^q L_x^{r_d}} \\
&\les \big\| \| \jb{\nabla}^{s + 1 - \g} \Psi_{\leq K} \|_{L^2 (\O)} \big\|_{L_T^q L_x^{r_d}} \\
&= \big\| \| S(\tau) \jb{\nabla}^{s + 1 - \g} \phi_{\leq K} e_k \|_{\l_k^2 L_\tau^2 ([0, t])} \big\|_{L_T^q L_x^{r_d}} \\
&\leq \big\| \| S(\tau) \jb{\nabla}^{s + 1 - \g} P_{\leq 1} (\phi_{\leq K} e_k) \|_{\l_k^2 L_\tau^2 ([0, t])} \big\|_{L_T^q L_x^{r_d}} \\
&\quad + \big\| \| S(\tau) \jb{\nabla}^{s + 1 - \g} P_{> 1} (\phi_{\leq K} e_k) \|_{\l_k^2 L_\tau^2 ([0, t])} \big\|_{L_T^q L_x^{r_d}},
\end{split}
\label{reg1}
\end{align}

\noi
where $P_{\leq 1}$ denotes the frequency projector onto $\{|\xi| \leq 1\}$ and $P_{> 1} = \Id - P_{\leq 1}$. For the low frequency part, we apply Minkowski's integral inequality, Bernstein's inequality, and the fact that $\frac{\sin (\tau |\xi|)}{|\xi|} \leq |\tau|$ to obtain
\begin{align}
\begin{split}
\big\| &\| S(\tau) \jb{\nabla}^{s + 1 - \g} P_{\leq 1} (\phi_{\leq K} e_k) \|_{\l_k^2 L_\tau^2 ([0, t])} \big\|_{L_T^q L_x^{r_d}} \\
&\leq T^{\frac 1q} \big\| \| S(\tau) \jb{\nabla}^{s + 1 - \g} P_{\leq 1} (\phi_{\leq K} e_k) \|_{L_x^{r_d}} \big\|_{\l_k^2 L_\tau^2 ([0, T])} \\
&\les T^{\frac 1q} \big\| \| S(\tau) \jb{\nabla}^{s + 1 - \g} P_{\leq 1} (\phi_{\leq K} e_k) \|_{L_x^{2}} \big\|_{\l_k^2 L_\tau^2 ([0, T])} \\
&\les T^\ta \big\| \| \jb{\nabla}^s \phi_{\leq K} e_k \|_{L_x^2} \big\|_{\l_k^2} \\
&= T^\ta \| \phi_{\leq K} \|_{\HS (L^2, H^s)}.
\end{split}
\label{reg2}
\end{align}

\noi
for some $\ta > 0$. For the high frequency part, by Minkowski's integral inequality, H\"older's inequality in $\tau$, and Lemma \ref{LEM:str}, we obtain
\begin{align}
\begin{split}
\big\| &\| S(\tau) \jb{\nabla}^{s + 1 - \g} P_{> 1} (\phi_{\leq K} e_k) \|_{\l_k^2 L_\tau^2 ([0, t])} \big\|_{L_T^q L_x^{r_d}} \\
&\leq T^{\ta} \big\| \| S(\tau) \jb{\nabla}^{s + 1 - \g} P_{> 1} (\phi_{\leq K} e_k) \|_{L_\tau^{\wt q} ([0, T]; L_x^{r_d})} \big\|_{\l_k^2} \\
&\les T^{\ta} \big\| \| |\nabla|^{-1 + \g} \jb{\nabla}^{s + 1 - \g} P_{> 1} (\phi_{\leq K} e_k) \|_{L^2_x} \big\|_{\l_k^2} \\
&\les T^\ta \| \phi_{\leq K} \|_{\HS (L^2, H^s)}.
\end{split}
\label{reg3}
\end{align}

\noi
Combining \eqref{reg1}, \eqref{reg2}, and \eqref{reg3}, we obtain
\begin{align*}
\big\| \| \Psi_{\leq K} \|_{L_T^q W_x^{s + 1 - \g, r_d}} \big\|_{L^p (\O)} \les T^\ta \| \phi_{\leq K} \|_{\HS (L^2, H^s)}.
\end{align*}

\noi
Similarly, we obtain that for $K_1, K_2 \in \N$ with $K_1 < K_2$,
\begin{align*}
\big\| \| \Psi_{\leq K_2} - \Psi_{\leq K_1} \|_{L_T^q W_x^{s + 1 - \g, r_d}} \big\|_{L^p (\O)} \les T^\ta \| \phi_{\leq K_2} - \phi_{\leq K_1} \|_{\HS (L^2, H^s)} \too 0
\end{align*}

\noi
as $K_1, K_2 \to \infty$. This then implies the convergence and \eqref{reg_goal} in part (ii).
\end{proof}

\begin{remark} \rm
\label{RMK:dtPsi}
One can use integration by parts to write
\begin{align*}
\Psi (t) = - \sum_{k \in \N} \int_0^t \beta_k (t') \frac{d}{ds} \Big|_{s = t'} \big( S(t - s) \phi e_k \big) dt'
\end{align*}

\noi
almost surely. This allows us to compute that
\begin{align*}
\dt \Psi (t) = \sum_{k \in \N} \int_0^t \dt S (t - t') \phi e_k d \beta_k (t')
\end{align*}

\noi
almost surely. As in Lemma \ref{LEM:Psi} (i), given $d \geq 1$, $T > 0$, and $\phi \in \HS (L^2, H^s)$ for some $s \in \R$, we can show that $\dt \Psi \in C ([0, T]; H^s (\R^d))$ almost surely and
\begin{align*}
\E \Big[ \sup_{0 \leq t \leq T} \| \dt \Psi (t) \|_{H^{s}}^p \Big] \leq C \| \phi \|_{\HS (L^2, H^s)}^p
\end{align*}

\noi
for any finite $p \geq 1$.
\end{remark}

\medskip
We now consider the stochastic convolution $\Psi$ in \eqref{Psi} on $\T^d$, which we denote by $\Phi := \Psi$ to avoid confusion. We recall that in the $\T^d$ setting, we assume that $\phi$ is a Fourier multiplier operator on $L^2 (\T^d)$: for all $n \in \Z^d$, we have
\begin{align*}
\phi e^{in \cdot x} = \ft \phi_n e^{in \cdot x}
\end{align*}

\noi
for some $\ft \phi_n \in \C$.
Thus, $\Phi$ is given by
\begin{align}
\Phi (t) = \int_0^t S_{\T^d} (t - t') \phi dW(t') = \sum_{n \in \Z^d} \int_0^t S_{\T^d} (t - t') (\ft \phi_n e^{in \cdot x}) d \beta_n (t').
\label{defPhi}
\end{align}

\noi
Here, the operator $S_{\T^d}$ has the same form as $S$ in \eqref{defS} but on the periodic domain $\T^d$ and $\be_n$ is now defined by $\be_n (t) = \langle \xi, \ind_{[0, t]} \cdot e^{in \cdot x} \rangle_{t, x}$, where $\langle \cdot, \cdot \rangle_{t, x}$ denotes the duality pairing on $\R_+ \times \T^d$. 

For later purposes of showing global well-posedness in the $\T^d$ setting, we also need a truncated periodized version of $\Phi$.
Let $\eta : \R^d \to [0, 1]$ be a smooth cutoff function such that $\eta \equiv 1$ on $[-2 \pi, 2 \pi]$ and $\eta \equiv 0$ outside of $[-4 \pi, 4 \pi]$. Given $R \geq 1$, we define
\begin{align}
\eta_R (x) = \eta \Big( \frac{x}{R} \Big).
\label{etaR}
\end{align}

\noi
Given $R \geq 1$, we define the following stochastic convolution on $\R^d$:
\begin{align}
\mathbf{\Phi}_R (t) &:= \sum_{n \in \Z^d} \int_0^t S(t - t') \big( \eta_R (x) \ft \phi_n e^{in \cdot x} \big) d \be_n (t'). \label{defPhiR}
\end{align}
Also, in view of Remark \ref{RMK:dtPsi}, we also also take a time derivative and obtain
\begin{align*}
\dt \mathbf{\Phi}_R (t) &:= \sum_{n \in \Z^d} \int_0^t \dt S(t - t') \big( \eta_R (x) \ft \phi_n e^{in \cdot x} \big) d \be_n (t').
\end{align*}

We now state the following lemma regarding the regularity of the above stochastic objects.
\begin{lemma}
\label{LEM:Phi}
Let $d \geq 1$, $T > 0$, and $\eps > 0$. 
Suppose that $\phi$ is a Fourier multiplier operator on $L^2 (\T^d)$ such that $\phi \in \HS (L^2 (\T^d), H^s (\T^d))$ for some $s > -1$. 

\smallskip \noi
\textup{(i)} 
Let $1 \leq r \leq \infty$. Then, we have $\Phi \in C([0, T]; W^{s + 1 - \eps, r} (\T^d))$ almost surely. More precisely, there exists $C = C(\o, \| \phi \|_{\HS (L^2, H^s)}, T) > 0$ such that
\begin{align*}
\| \Phi (t) \|_{L_T^\infty W_x^{s + 1 - \eps, r} (\T^d)} \leq C.
\end{align*}

\smallskip \noi
\textup{(ii)}
Let $1 \leq q < \infty$ and $2 \leq r < \infty$. Then, for any $R \geq 1$, we have almost surely $\mathbf{\Phi}_R \in L^q([0, T]; W^{s + 1, r} (\R^d))$, $\mathbf{\Phi}_R \in C([0, T]; W^{s + 1 - \eps, \infty} (\R^d))$, and $\dt \mathbf{\Phi}_R \in C([0, T]; W^{s - \eps, \infty} (\R^d))$. More precisely, there exists $C = C(\o, \| \phi \|_{\HS (L^2, H^s)}, R, T) > 0$ such that 
\begin{align}
\| \mathbf{\Phi}_R \|_{L^q_T W_x^{s + 1, r} (\R^d)} &\leq C, \label{Phi1} \\
\| \mathbf{\Phi}_R \|_{L^\infty_T W_x^{s + 1 - \eps, \infty} (\R^d)} &\leq C, \label{Phi2} \\
\| \dt \mathbf{\Phi}_R \|_{L^\infty_T W_x^{s - \eps, \infty} (\R^d)} &\leq C. \label{Phi3}
\end{align}
\end{lemma}

\begin{proof}
The proof of (i) follows from straightforward modifications of  \cite[Lemma 3.1]{GKO2} or \cite[Proposition 2.1]{GKO}.

For (ii), we first consider \eqref{Phi1}. We note that for any $s > -1$, $t > 0$, and $n \in \Z^d$, we have
\begin{align}
\jb{\xi}^{s + 1} \bigg| \frac{\sin (t |\xi|)}{|\xi|} \bigg| \les \max (1, t) \jb{\xi}^{s} \leq \max(1, t) \jb{\xi - n}^{|s|} \jb{n}^{s}.
\label{sin}
\end{align}

\noi
Using \eqref{sin}, Hausdorff-Young's inequality, and H\"older's inequality, we obtain
\begin{align}
\begin{split}
\Big\| &\jb{\nb}^{s + 1} \frac{\sin (t |\nb|)}{|\nb|} (\eta_R e^{in \cdot x}) \Big\|_{L_x^r (\R^d)} \\
&\leq \max(1, t) \jb{n}^{s} \big\| \jb{\xi - n}^{|s|} \ft \eta_R (\xi - n) \|_{L_\xi^{r'} (\R^d)} \\
&\les \max(1, t) \jb{n}^{s} \big\| \jb{\xi - n}^{|s| + \frac{d}{2}} \ft \eta_R (\xi - n) \|_{L_\xi^{2} (\R^d)} \\
&\les \max(1, t) \jb{n}^{s} \| \eta_R \|_{H^{|s| + \frac d2} (\R^d)}.
\end{split}
\label{sin2}
\end{align}
Thus, letting $p \geq \max (q, r)$, by Minkowski's integral inequality, the Gaussian hypercontractivity, the Ito isometry, Minkowski's integral inequality again, and \eqref{sin2}, we can compute
\begin{align*}
\big\| \| &\mathbf{\Phi}_R \|_{L_T^q W_x^{s + 1 , r} (\R^d) } \big\|_{L^p (\O)} \\
&\leq \big\| \| \jb{\nb}^{s + 1} \mathbf{\Phi}_R \|_{L^p (\O)} \big\|_{L^q_T L_x^r (\R^d)} \\
&\les \big\| \| \jb{\nb}^{s + 1} \mathbf{\Phi}_R \|_{L^2 (\O)} \big\|_{L^q_T L_x^r (\R^d)} \\
&= \bigg\|  \Big\| \jb{\nb}^{s + 1} \frac{\sin (t' |\nb|)}{|\nb|} \big( \eta_R \ft \phi_n e^{in \cdot x} \big) \Big\|_{\l_n^2 L_{t'}^2 ([0, t])} \bigg\|_{L^q_T L_x^r (\R^d)} \\
&\les T^{\frac 1q} \bigg\| \Big\| \jb{\nb}^{s + 1} \frac{\sin (t' |\nb|)}{|\nb|} \big( \eta_R e^{in \cdot x} \big) \Big\|_{L_x^r (\R^d)} \cdot \ft \phi_n \bigg\|_{\l_n^2 L_{t'}^2 ([0, T])} \\
&\les T^{\frac 12 + \frac 1q} \| \eta_R \|_{H^{|s| + \frac d2} (\R^d)} \big\| \jb{n}^s \ft \phi_n \big\|_{\l_n^2} \\
&= T^{\frac 12 + \frac 1q} \| \eta_R \|_{H^{|s| + \frac d2} (\R^d)} \| \phi \|_{\HS (L^2, H^s)}.
\end{align*}

\noi
This shows the desired estimate \eqref{Phi1} and the fact that $\mathbf{\Phi}_R$ lies in $L^q ([0, T]; W^{s + 1 , r} (\R^d))$ almost surely.

For the estimate \eqref{Phi2}, we can apply the Sobolev embedding $L^\infty (\R^d) \embeds W^{\eps, r_1} (\R^d)$ with $0 < \eps \ll 1$ and $2 \leq r_1 < \infty$ satisfying $\eps r_1 > d$ and then use the above steps to obtain the desired result. The fact that $\mathbf{\Phi}_R$ is continuous in time follows from adapting the above modifications to the proof of \cite[Proposition 2.1]{GKO} using Kolmogorov's continuity criterion. The steps for establishing the estimate \eqref{Phi3} are similar.
\end{proof}

\section{Local well-posedness of the energy-critical stochastic NLW}
\label{SEC:LWP}
In this section, we briefly go over local well-posedness of the defocusing energy-critical SNLW \eqref{SNLW} on $\R^d$. We first show the following local well-posedness result in a slightly more general setting. Recall that the operator $V(t)$ is as defined in \eqref{defS} and the $X(I)$-norm is as defined in \eqref{defX}.
\begin{proposition}
\label{PROP:LWP}
Let $d \geq 3$ and $(u_0, u_1) \in \dot{\H}^1 (\R^d)$. Then, there exists $0 < \eta \ll 1$ such that if 
\begin{align}
\| V(t - t_0) (u_0, u_1) \|_{X(I)} \leq \eta \quad \text{and} \quad \| f \|_{X(I)} \leq \eta
\label{cond1}
\end{align}

\noi
for some time interval $I = [t_0, t_1] \subset \R$, then the Cauchy problem
\begin{align}
\begin{cases}
\dt^2 v - \Dl v + \NN (v + f) = 0 \\
(v, \dt v)|_{t = t_0} = (u_0, u_1)
\end{cases}
\label{NLWvf}
\end{align}

\noi
with $\NN(u) = |u|^{\frac{4}{d - 2}} u$ admits a unique solution $(v, \dt v) \in C(I; \dot{\H}^1 (\R^d))$, which satisfies
\begin{align*}
\| v \|_{X(I)} \leq 3 \eta.
\end{align*} 
Here, the uniqueness of $v$ holds in the set
\begin{align*}
\{ v \in X(I): \| v \|_{X(I)} \leq 3 \eta \}.
\end{align*}

\end{proposition}
\begin{proof}
By writing \eqref{NLWvf} in the Duhamel formulation, we have
\begin{align}
v (t) = \G [v] (t) := V (t - t_0) (u_0, u_1) - \int_{t_0}^t S (t - t') \NN (v + f) (t') dt',
\label{Duhv1}
\end{align}

\noi
where $S(t)$ is as defined in \eqref{defS}. We would like to run the contraction argument on the ball
\begin{align*}
B_{\eta, I} := \{ v \in X (I): \| v \|_{X(I)} \leq 3\eta \},
\end{align*}

\noi
where $\eta > 0$ is to be chosen later.

Suppose that $\| (u_0, u_1) \|_{\dot{\H}} \leq A$ for some $A > 0$. 
Then, by the Strichartz estimate (Lemma \ref{LEM:str}) and \eqref{cond1}, for any $v \in B_{a, I}$, we have
\begin{align*}
\| \G [v] \|_{X(I)} &\leq \| V(t - t_0) (u_0, u_1) \|_{X(I)} + C_1 \| v \|_{X(I)}^{\frac{d + 2}{d - 2}} + C_1 \| f \|_{X(I)}^{\frac{d + 2}{d - 2}} \\
&\leq \eta + C_1 (3 \eta)^{\frac{d + 2}{d - 2}} + C_1 \eta^{\frac{d + 2}{d - 2}} \\
&\leq 3\eta
\end{align*}

\noi
for some constant $C_1 > 0$, where in the last inequality we choose $0 < \eta \ll 1$ in such a way that
\begin{align*}
C_1 (3 \eta)^{\frac{d + 2}{d - 2}} \leq \eta \quad \text{and} \quad C_1 \eta^{\frac{d + 2}{d - 2}} \leq \eta.
\end{align*}

\noi
Also, by the Strichartz estimate (Lemma \ref{LEM:str}) and the fundamental theorem of calculus, for any $v_1, v_2 \in B_{\eta, I}$, we have
\begin{align*}
\| \G [v_1] - \G [v_2] \|_{X (I)} &\leq C_2 \| \NN (v_1) - \NN (v_2) \|_{L_I^1 L_x^2} \\
&= C_2 \bigg\| \int_0^1 \NN' \big( v_2 + \al (v_1 - v_2) \big) (v_1 - v_2) d \al \bigg\|_{L_I^1 L_x^2} \\
&\leq C_2 \Big( \| v_1 \|_{X (I)}^{\frac{4}{d - 2}} + \| v_2 \|_{X (I)}^{\frac{4}{d - 2}} \Big) \| v_1 - v_2 \|_{X (I)} \\
&\leq 2C_2 (3 \eta)^{\frac{4}{d - 2}} \| v_1 - v_2 \|_{X (I)} \\
&\leq \frac 12 \| v_1 - v_2 \|_{X (I)}
\end{align*}

\noi
for some constant $C_2 > 0$, where in the last inequality we further shrink $0 < \eta \ll 1$ if possible so that
\begin{align*}
2C_2 (3 \eta)^{\frac{4}{d - 2}} \leq \frac 12.
\end{align*}

\noi
Thus, $\G$ is a contraction from $B_{\eta, I}$ to itself, so that the desired local well-posedness of the equation \eqref{NLWvf} follows. The fact that $(v, \dt v) \in C(I; \dot{\H}^1 (\R^d))$ then follows easily from the Duhamel formulation \eqref{Duhv1}, the Strichartz estimate (Lemma \ref{LEM:str}), \eqref{cond1}, and the fact that $v \in B_{\eta, I}$.
\end{proof}

As a consequence of local well-posedness in Proposition \ref{PROP:LWP}, we have the following blowup alternative.
\begin{lemma}
\label{LEM:blow}
Let $d \geq 3$ and $(u_0, u_1) \in \dot{\H}^1 (\R^d)$. Let $f$ be a function such that $\| f \|_{X([0, T])} < \infty$ for any $T > 0$. Let $T^* = T^* (u_0, u_1, f) > 0$ be the forward maximal time of existence of the following Cauchy problem for $v$:
\begin{align*}
\begin{cases}
\dt^2 v - \Dl v + \NN (v + f) = 0 \\
(v, \dt v)|_{t = t_0} = (u_0, u_1).
\end{cases}
\end{align*}

\noi
Then, we have either
\begin{align*}
T^* = \infty \quad \text{or} \quad \lim_{T \nearrow T^*} \| v \|_{X([0, T])} = \infty.
\end{align*}
\end{lemma}

We now go back to the defocusing energy-critical SNLW \eqref{SNLW} on $\R^d$. By writing $u = v + \Psi$ with $\Psi$ being the stochastic convolution as defined in \eqref{Psi2}, we see that $v$ satisfies
\begin{align}
\begin{cases}
\dt^2 v - \Dl v + \NN (v + \Psi) = 0 \\
(v, \dt v)|_{t = t_0} = (u_0, u_1).
\end{cases}
\label{NLWvPsi}
\end{align}

\noi
We also note that by Lemma \ref{LEM:Psi}, given $\phi \in \HS (L^2 (\R^d), L^2 (\R^d))$, we have
\begin{align}
\| \Psi \|_{X([t_0, t_0 + \tau])} \leq C(\o) \tau^\ta \| \phi \|_{\HS (L^2, L^2)}
\label{Psi_bdd} 
\end{align}

\noi
almost surely for any $t_0 \geq 0$, $\tau > 0$, and some $\ta > 0$. This norm can be made arbitrarily small if we make $\tau > 0$ to be small. Thus, the local well-posedness result (Proposition \ref{PROP:LWP}) and the blowup alternative (Lemma \ref{LEM:blow}) apply to the equation \eqref{NLWvPsi} by letting $f = \Psi$.

\section{Proof of global well-posedness}
In this section, we show the proofs of the two global well-posedness results: Theorem \ref{THM:R} and Theorem \ref{THM:T}.

\subsection{Global well-posedness on $\R^{d}$}
\label{SUBSEC:GWP_R}
In this subsection, we prove Theorem \ref{THM:R}, global well-posedness of the defocusing energy-critical SNLW \eqref{SNLW} on $\R^d$. As mentioned in Section \ref{SEC:intro}, the proof is based on the perturbation lemma (Lemma \ref{LEM:pert}). However, to carry out the perturbation argument, we first need to show an a priori energy bound using Ito's lemma. In order to justify the use of Ito's lemma, we need to perform an approximation procedure for the defocusing energy-critical SNLW \eqref{SNLW} as below.

Given $N \in \N$, we denote $P_{\les N}$ to be a smooth frequency projection onto $\{|\xi| \leq N \}$. We consider the following truncated defocusing energy-critical SNLW:
\begin{align}
\begin{cases}
\dt^2 u_{\les N} - \Dl u_{\les N} + P_{\les N} \mathcal{N} (u_{\les N}) = P_{\les N} \phi \xi \\
(u_{\les N}, \dt u_{\les N})|_{t = 0} = (P_{\les N} u_0, P_{\les N} u_1),
\end{cases}
\label{SNLWt}
\end{align}

\noi
where $\NN (u) = |u|^{\frac{4}{d - 2}} u$. 
We define the truncated stochastic convolution $\Psi_{\les N}$ by \eqref{Psi2} with $\phi$ replaced by $P_{\les N} \phi$. Due to the boundedness of $P_{\les N}$, we can easily see from Lemma \ref{LEM:Psi} that
\begin{align*}
\| \Psi_{\les N} \|_{X([t_0, t_0 + T])} \leq C(\o) T^\ta \| \phi \|_{\HS (L^2, L^2)} 
\end{align*}

\noi
almost surely for any $t_0 \geq 0$, $T > 0$, and some $\ta > 0$. Thus, again due to the boundedness of $P_{\les N}$, we see that the local well-posedness result (Proposition \ref{PROP:LWP}) and the blowup alternative (Lemma \ref{LEM:blow}) apply to the following equation for $v_{\les N} := u_{\les N} - \Psi_{\les N}$:
\begin{align*}
\begin{cases}
\dt^2 v_{\les N} - \Dl v_{\les N} + P_{\les N} \mathcal{N} (v_{\les N} + \Psi_{\les N}) = 0 \\
(v_{\les N}, \dt v_{\les N})|_{t = 0} = (P_{\les N} u_0, P_{\les N} u_1).
\end{cases}
\end{align*}

\noi
This shows that the truncated defocusing energy-critical SNLW \eqref{SNLWt} is locally well-posed.

Let us now show the following lemma regarding the convergence of $u_{\les N}$ to the solution $u$ to SNLW \eqref{SNLW}. 
\begin{lemma}
\label{LEM:uN}
Let $d \geq 3$, $\phi \in \HS (L^2, L^2)$, and $(u_0, u_1) \in \H^1$. Then, the following holds true almost surely. Assume that $u$ is a solution to the defocusing energy-critical SNLW \eqref{SNLW} on $[0, T]$ for some $T > 0$, and assume that $u_{\les N}$ is a solution to the truncated defocusing energy-critical SNLW \eqref{SNLWt} on $[0, T_N]$ for some $T_N > 0$. Also, let $R > 0$ be such that $\| u \|_{X ([0, T])} \leq R$, where the $X([0, T])$-norm is as defined in \eqref{defX}. Then, by letting $S_N = \min (T, T_N)$, we have
\begin{align*}
\| \vec{u} - \vec{u}_{\les N} \|_{C([0, S_N]; \dot{\H}_x^1)} \too 0
\end{align*}
and
\begin{align*}
\| u - u_{\les N} \|_{X([0, S_N] \times \R^d)} \too 0
\end{align*}

\noi
as $N \to \infty$.
\end{lemma}

\begin{remark} \rm
In Lemma \ref{LEM:uN}, due to the lack of global well-posedness of the truncated equation \eqref{SNLWt} at this point, we need to assume that the existence time $T_N$ of \eqref{SNLWt} depends on $N$. One can avoid this issue by inserting a $H^1$-norm truncation as in \cite{dBD}, so that the nonlinearity becomes Lipschitz and hence one has global well-posedness for the truncated equation. In this paper, we choose to proceed with the existence time $T_N$ dependent on $N$, which turns out to be harmless at a later point (see the proof of Proposition \ref{PROP:energy}).
\end{remark}

\begin{proof}
Let $0 < \eta \ll 1$ be chosen later and $0 < \eps \ll \eta$. By slight modifications of Lemma \ref{LEM:Psi}, we have that for almost sure $\o \in \O$,
\begin{align}
\| \Psi - \Psi_{\les N} \|_{X ([0, T])} + \| \Psi - \Psi_{\les N} \|_{L_{T}^\infty H_x^1} < \eps
\label{uN2}
\end{align}

\noi
given $N \geq N_0 (\o, \eps)$. 

We now divide the interval $[0, T]$ into $J = J(R, \eta)$ many subintervals $I_j = I_j (\o) = [t_j, t_{j + 1}]$ with $0 = t_0 < t_1 < \cdots < t_{J - 1} < t_J = T$ such that
\begin{align}
\| u \|_{X(I_j)} \leq \eta
\label{uN1}
\end{align}

\noi
for $j = 0, 1, \dots, J - 1$.

As in the proof of Proposition \ref{PROP:LWP}, by using the Duhamel formulation \eqref{Duh}, the Strichartz estimate (Lemma \ref{LEM:str}), and the fundamental theorem of calculus, we obtain
\begin{align}
\begin{split}
\| u - u_{\les N} \|_{X ([0, \tau])} &\les \| u (0) - u_{\les N} (0) \|_{\dot{\H}^1} \\
&\quad + \big( \| u \|_{X ([0, \tau])} + \| u_{\les N} \|_{X ([0, \tau])} \big)^{\frac{4}{d - 2}} \| u - u_{\les N} \|_{X ([0, \tau])} \\
&\quad + \| (\Id - P_{\les N}) \NN (u) \|_{L_t^1([0, \tau]; L_x^2)} + \| \Psi - \Psi_{\les N} \|_{X ([0, \tau])}
\end{split}
\label{uN3}
\end{align}

\noi
for any $0 \leq \tau \leq \min (t_1, T_N)$. Then, from \eqref{uN3}, by the Lebesgue dominated convergence theorem applied to $(\Id - P_{\les N}) u_0$, $(\Id - P_{\les N}) u_1$, and $(\Id - P_{\les N}) \NN (u)$ along with \eqref{uN2} and \eqref{uN1}, we have
\begin{align*}
\| u - u_{\les N} \|_{X ([0, \tau])} < C\eps + C \big( \eta + \| u_{\les N} \|_{X ([0, \tau])} \big)^{\frac{4}{d - 2}} \| u - u_{\les N} \|_{X ([0, \tau])}
\end{align*}

\noi
for some absolute constant $C > 0$, given $N \geq N_0 (\o, \eps, u)$ sufficiently large. By taking $\eta > 0$ to be sufficiently small (independent of $\eps$), we can then use a standard continuity argument to get
\begin{align}
\| u - u_{\les N} \|_{X (I_0 \cap [0, T_N])} < 2C \eps
\label{uN4}
\end{align}
and
\begin{align}
\| u_{\les N} \|_{X (I_0 \cap [0, T_N])} \leq 2 \eta.
\label{uN5}
\end{align}

\noi
Similar as above, we use the Duhamel formulation \eqref{Duh}, the Strichartz estimate (Lemma \ref{LEM:str}), the Lebesgue dominated convergence theorem, \eqref{uN2}, \eqref{uN1}, and \eqref{uN5}, we obtain
\begin{align*}
\| \vec u - \vec u_{\les N} \|_{L^\infty(I_0 \cap [0, T_N]; \dot{\H}_x^1)} < C \eps + C \eta^{\frac{4}{d - 2}} \| u - u_{\les N} \|_{X (I_0 \cap [0, T_N])}
\end{align*}

\noi
given $N \geq N_0 (\o, \eps, u)$. By \eqref{uN4} and taking $\eta > 0$ to be sufficiently small (independent of $\eps$), we then obtain
\begin{align*}
\| \vec u - \vec u_{\les N} \|_{L^\infty(I_0 \cap [0, T_N]; \dot{\H}_x^1)} < 2 C \eps.
\end{align*}

\noi
In particular, if $t_1 \leq T_N$, we have
\begin{align*}
\| \vec u (t_1) - \vec u_{\les N} (t_1) \|_{\dot{\H}^1} < 2 C \eps.
\end{align*}

We now repeat the above arguments on $I_1$ to obtain
\begin{align*}
\| u - u_{\les N} \|_{X (I_1 \cap [0, T_N])} < 4 C^2 \eps 
\qquad \text{and} \qquad
\| \vec u - \vec u_{\les N} \|_{L^\infty(I_1 \cap [0, T_N]; \dot{\H}_x^1)} < 4 C^2 \eps.
\end{align*}

\noi
By applying the above arguments repetitively, we obtain that for $j = 0, 1, \dots, J - 1$,
\begin{align*}
\| u - u_{\les N} \|_{X (I_j \cap [0, T_N])} < (2 C)^{j + 1} \eps 
\qquad \text{and} \qquad
\| \vec u - \vec u_{\les N} \|_{L^\infty(I_j \cap [0, T_N]; \dot{\H}_x^1)} < (2 C)^{j + 1} \eps.
\end{align*} 

\noi
Thus, we have
\begin{align*}
&\| \vec u - \vec u_{\les N} \|_{L^\infty([0, S_N]; \dot{\H}_x^1)} \leq \sum_{j = 0}^{J - 1} \| \vec u - \vec u_{\les N} \|_{L^\infty(I_j \cap [0, T_N]; \dot{\H}_x^1)} \leq (2 C)^J \eps, \\
&\| u - u_{\les N} \|_{X ([0, S_N])} \leq \sum_{j = 0}^{J - 1} \| u - u_{\les N} \|_{X (I_j \cap [0, T_N])} \leq (2 C)^J \eps.
\end{align*}

\noi
Since $\eps > 0$ can be arbitrarily small and $J$ depends only on $R > 0$ and an absolute constant $\eta > 0$, we can conclude the desired convergence results.

\end{proof}

We now show the following a priori bound on the energy $E$ as defined in \eqref{energy}.
\begin{proposition}
\label{PROP:energy}
Let $d \geq 3$, $\phi \in \HS (L^2, L^2)$, and $(u_0, u_1) \in \dot{\H}^1 (\R^d)$. Let $u$ be the solution to the defocusing energy-critical SNLW \eqref{SNLW} with $(u, \dt u)|_{t = 0} = (u_0, u_1)$ and let $T^* = T^* (\o, u_0, u_1)$ be the forward maximal time of existence. Then, given any $T_0 > 0$, there exists $C = C( \| (u_0, u_1) \|_{\dot{\H}^1}, \| \phi \|_{\HS (L^2, L^2)}, T_0 ) > 0$ such that for any stopping time $T$ with $0 < T < \min (T^* , T_0)$ almost surely, we have
\begin{align*}
\E \Big[ \sup_{0 \leq t \leq T} E (\vec u (t)) \Big] \leq C.
\end{align*}
\end{proposition}

\begin{proof}
We write the energy $E$ in \eqref{energy} as
\begin{align*}
E \big( u_1 (t), u_2 (t) \big) = \frac 12 \int_{\R^d} |u_2 (t)|^2 dx + \frac 12 \int_{\R^d} |\nabla u_1 (t)|^2 dx + \frac{d - 2}{2d} \int_{\R^d} |u_1 (t)|^{\frac{2d}{d - 2}} dx.
\end{align*}

\noi
A direct computation yields
\begin{align}
E' \big( u_1 (t), u_2 (t) \big) (v_1, v_2) =  \int_{\R^d} \Big( u_2 (t) v_2 + \nb u_1 (t) \cdot \nb v_1 + |u_1 (t)|^{\frac{4}{d - 2}} u_1 (t) v_1 \Big) dx
\label{E1}
\end{align}

\noi
and
\begin{align}
\begin{split}
E'' \big( u_1 (t), &u_2 (t) \big) \big( (v_1, v_2), (w_1, w_2) \big) \\
&= \int_{\R^d} v_2 w_2 dx + \int_{\R^d} \nb v_1 \cdot \nb w_1 dx + \frac{d + 2}{d - 2} \int_{\R^d} |u_1 (t)|^{\frac{4}{d - 2}} v_1 w_1 dx
\end{split}
\label{E2}
\end{align}

Given $R > 0$, we define the stopping time
\begin{align*}
T_1 = T_1 (R) := \inf \big\{ \tau \geq 0: \| u \|_{X ([0, \tau])} \geq R \big\},
\end{align*}

\noi
where the $X ([0, \tau])$-norm is as defined in \eqref{defX}. We also set $T_2 := \min (T, T_1)$. Note that we have
\begin{align*}
\| u \|_{X ([0, T])} < \infty
\end{align*}

\noi
almost surely in view of the blowup alternative in Lemma \ref{LEM:blow}, so that we have $T_2 \nearrow T$ almost surely as $R \to \infty$.

Let $u_{\les N}$ be the solution to the truncated defocusing energy-critical SNLW \eqref{SNLWt} with maximal time of existence $T_N^* = T_N^* (\o)$. By Lemma \ref{LEM:uN}, we can deduce that there exists an almost surely finite number $N_0 (\o) \in \N$ such that $T_N^* \geq T_2$ for any $N \geq N_0 (\o)$. Indeed, suppose not, then there exists $\O_0 \subset \O$ with positive measure such that for all $\o \in \O_0$, there exists a sequence of increasing numbers $\{ N_j (\o) \}_{j \in \N} \subset \N$ such that $T_{N_j (\o)}^* < T_2$. By the blowup alternative (Lemma \ref{LEM:blow}), we know that for all $\o \in \O_0$ and $j \in \N$, there exists $T_{N_j (\o)} < T_2$ such that
\begin{align*}
\| u_{\les N_j (\o)} \|_{X([0, T_{N_j (\o)}])} > 2R.
\end{align*}

\noi
This contradicts with the convergence of $u_{\les N_j (\o)}$ to $u$ in $X ([0, T_{N_j (\o)}])$ as $N_j (\o) \to \infty$ in Lemma \ref{LEM:uN}, given that we have
\begin{align*}
\| u \|_{X ([0, T_2])} \leq R.
\end{align*}

We can now work with \eqref{SNLWt} on $[0, T_2]$. Note that we can write \eqref{SNLWt} in the following Ito formulation:
\begin{align*}
d \begin{pmatrix}
u_{\les N} \\ \dt u_{\les N}
\end{pmatrix}
=
\begin{pmatrix}
0 & 1 \\ \Dl & 0
\end{pmatrix}
\begin{pmatrix}
u_{\les N} \\ \dt u_{\les N}
\end{pmatrix} dt 
+
\begin{pmatrix}
0 \\ - P_{\les N} \NN (u_{\les N})
\end{pmatrix} dt
+
\begin{pmatrix}
0 \\ P_{\les N} \phi dW
\end{pmatrix}.
\end{align*}

\noi
By Ito's lemma (see \cite[Theorem 4.32]{DPZ}) along with \eqref{E1} and \eqref{E2}, we obtain
\begin{align}
\begin{split}
E \big( u_{\les N} (t), &\dt u_{\les N} (t) \big) = E (P_{\les N} u_0, P_{\les N} u_1) + \sum_{k \in \N} \int_0^t \int_{\R^d} \dt u_{\les N} (t') P_{\les N} \phi e_k dx d \be_k (t') \\
& + \int_0^t \int_{\R^d} \dt u_{\les N} (t') (\Id - P_{\les N}) \NN (u_{\les N}) dx dt' + t \| P_{\les N} \phi \|_{\HS (L^2, L^2)}^2
\end{split}
\label{EN}
\end{align} 

\noi
for $0 < t < T_2$.

For the third term on the right-hand-side of \eqref{EN}, by H\"older's inequalities, the Lebesgue dominated convergence theorem applied to $(\Id - P_{\les N}) \NN (u)$, and Lemma \ref{LEM:uN}, we have 
\begin{align}
\begin{split}
\bigg| \int_0^t &\int_{\R^d} \dt u_{\les N} (t') (\Id - P_{\les N}) \NN (u_{\les N}) dx dt' \bigg| \\
&\les \| \dt u_{\les N} \|_{L_{T_2}^\infty L_x^2} \| (\Id - P_{\les N}) \NN (u_{\les N}) \|_{L_{T_2}^1 L_x^2} \\
&\les \| \dt u_{\les N} \|_{L_{T_2}^\infty L_x^2} \Big( \| (\Id - P_{\les N}) \NN (u) \|_{L_{T_2}^1 L_x^2} + \| \NN (u) - \NN (u_{\les N}) \|_{L_{T_2}^1 L_x^2} \Big) \\
&\les \| \dt u_{\les N} \|_{L_{T_2}^\infty L_x^2} \Big( \| (\Id - P_{\les N}) \NN (u) \|_{L_{T_2}^1 L_x^2} \\
&\quad + \big( \| u \|_{X ([0, T_2])} + \| u_{\les N} \|_{X ([0, T_2])} \big)^{\frac{4}{d - 2}} \| u - u_{\les N} \|_{X ([0, T_2])} \Big) \\
&\too 0
\end{split}
\label{EN1}
\end{align}

\noi
almost surely, as $N \to \infty$. Thus, by applying the Burkholder-Davis-Gundy inequality to \eqref{EN} along with \eqref{EN1}, H\"older's inequality, and Cauchy's inequality, we obtain
\begin{align*}
\begin{split}
\E \Big[ &\sup_{0 \leq t \leq T_2} E \big( u_{\les N} (t), \dt u_{\les N} (t) \big) \Big] \\
&\leq E (P_{\les N} u_0, P_{\les N} u_1) + C \E \bigg[ \bigg( \sum_{k \in \N} \int_0^{T_2} \bigg| \int_{\R^d} \dt u_{\les N} (t') P_{\les N} \phi e_k dx \bigg|^2 dt' \bigg)^{1/2} \bigg] \\
&\quad + \eps + T_2 \| \phi \|_{\HS (L^2, L^2)}^2 \\
&\leq C E (u_0, u_1) + C T_2^{\frac 12} \E \big[ \| \dt u_{\les N} \|_{L_{T_2}^\infty L_x^2}  \| \phi \|_{\HS (L^2, L^2)} \big] + \eps + T_2 \| \phi \|_{\HS (L^2, L^2)}^2 \\
&\leq \frac 12 \E \Big[ \sup_{0 \leq t \leq T_2} E \big( u_{\les N} (t), \dt u_{\les N} (t) \big) \Big] + C E (u_0, u_1) + C T_2 \| \phi \|_{\HS (L^2, L^2)}^2 + \eps
\end{split}
\end{align*}

\noi
for some absolute constant $C > 0$ and $\eps > 0$ arbitrarily small given $N \geq N_0 (\o, \eps, u)$ sufficiently large. This shows that
\begin{align}
\E \Big[ \sup_{0 \leq t \leq T_2} E \big( \vec u_{\les N} (t) \big) \Big] \leq C E (u_0, u_1) + C T_2 \| \phi \|_{\HS (L^2, L^2)}^2.
\label{EN2}
\end{align}
\noi
By Fatou's lemma and \eqref{EN2}, we have
\begin{align*}
\E \Big[ \sup_{0 \leq t \leq T_2} E \big( \vec u (t) \big) \Big] \leq \liminf_{N \to \infty} \E \Big[ \sup_{0 \leq t \leq T_2} E \big( \vec u_{\les N} (t) \big) \Big] \leq C
\end{align*}

\noi
for some constant $C = C( \| (u_0, u_1) \|_{\dot{\H}^1}, \| \phi \|_{\HS (L^2, L^2)}, T_0 )$. In view of the almost sure convergence of $T_2$ to $T$, we then obtain the desired energy bound using Fatou's lemma.
\end{proof}

We are now ready to prove Theorem \ref{THM:R}.
\begin{proof}[Proof of Theorem \ref{THM:R}]
Let $u$ be a local-in-time solution to the defocusing energy-critical SNLW \eqref{SNLW} given by Proposition \ref{PROP:LWP}. We let $v = u - \Psi$, where $v$ satisfies
\begin{align}
\begin{cases}
\dt^2 v - \Dl v + \NN (v + \Psi) = 0 \\
(v, \dt v)|_{t = 0} = (u_0, u_1),
\end{cases}
\label{SNLMv}
\end{align}

\noi
where $\NN (u) = |u|^{\frac{4}{d - 2}} u$. 
Given $T > 0$, if we suppose that the solution $u$ exists on $[0, T]$, then by the energy bound in Proposition \ref{PROP:energy}, we have
\begin{align*}
\sup_{0 \leq t \leq T} E (\vec u (t)) \leq C (\o, \| (u_0, u_1) \|_{\dot{\H}^1}, \| \phi \|_{\HS (L^2, L^2)}, T) < \infty.
\end{align*}

\noi
Then, by Lemma \ref{LEM:Psi} and Remark \ref{RMK:dtPsi}, we obtain
\begin{align}
\begin{split}
\sup_{0 \leq t \leq T} E (\vec v (t)) &\leq C \sup_{0 \leq t \leq T} E (\vec u (t)) + C \sup_{0 \leq t \leq T} E (\vec \Psi (t)) \\
&\leq C (\o, \| (u_0, u_1) \|_{\dot{\H}^1}, \| \phi \|_{\HS (L^2, L^2)}, T) \\
&= : R
\end{split}
\label{Ev}
\end{align}

Given a target time $T > 0$, we pick any $t_0 \in [0, T)$ and suppose that the solution $v$ to \eqref{SNLMv} has already been constructed on $[0, t_0]$. Our goal is to show the existence of a unique solution $v$ to \eqref{SNLMv} on $[t_0, t_0 + \tau] \cap [0, T]$ with $\tau > 0$ independent of $t_0$. In this way, we can iterate the argument so that a global solution $v$ to \eqref{SNLMv} can be constructed on $[0, T]$, which then concludes the proof of Theorem \ref{THM:R}. 

Let $w$ be the global solution to the deterministic defocusing energy-critical NLW:
\begin{align}
\dt^2 w - \Dl w + |w|^{\frac{4}{d - 2}} w = 0
\label{NLWw}
\end{align}

\noi
with $(w, \dt w)|_{t = t_0} = \vec v (t_0)$. By \eqref{Ev}, we have $\| \vec w (t_0) \|_{\dot{\H}^1} \leq R$. By Lemma \ref{LEM:bdd}, we have
\begin{align}
\| w \|_{X ([t_0, T])} \leq C (R) < \infty.
\label{bdd_w}
\end{align}

\noi
Let $0 < \eta \ll 1$ be chosen later. By \eqref{bdd_w}, we can divide the interval $[t_0, T]$ into $J = J(R, \eta)$ many subintervals $I_j = [t_j, t_{j + 1}]$ such that
\begin{align}
\| w \|_{X (I_j)} \leq \eta
\label{bdd_wj}
\end{align} 

\noi
for $j = 0, 1, \dots, J - 1$. Recalling the operators $V(t)$ and $S(t)$ in \eqref{defS}, by using the Duhamel formulation of \eqref{NLWw}, the Strichartz estimate (Lemma \ref{LEM:str}), and \eqref{bdd_wj}, we have
\begin{align}
\begin{split}
\big\| V (t - t_j) \big( w (t_j), \dt w (t_j) \big) \big\|_{X(I_j)} 
&= \bigg\| w(t) + \int_{t_j}^t S(t - t') \NN (w) (t') dt' \bigg\|_{X(I_j)} \\
&\leq \| w \|_{X (I_j)} + C \| w \|_{X (I_j)}^{\frac{d + 2}{d - 2}} \\
&\leq \eta + C \eta^{\frac{d + 2}{d - 2}} \\
&\leq 2 \eta
\end{split}
\label{Sw}
\end{align}

\noi
for $j = 0, 1, \dots, J - 1$, given that $\eta > 0$ is sufficiently small.

We now consider $v$ on the first interval $I_0$, which we can assume that $I_0 \subseteq [t_0, t_0 + \tau]$ for some $\tau > 0$ to be determined later. Since $\vec v (t_0) = \vec w (t_0)$, by \eqref{Sw} we have
\begin{align}
\big\| V (t - t_0) \big( v (t_0), \dt v (t_0) \big) \big\|_{X(I_0)} 
= \big\| V (t - t_0) \big( w (t_0), \dt w (t_0) \big) \big\|_{X(I_0)} \leq 2 \eta.
\label{Sv}
\end{align}

\noi
Using Proposition \ref{PROP:LWP} along with \eqref{Sv} and \eqref{Psi_bdd} (with $\tau > 0$ sufficiently small such that $C(\o) \tau^\ta \| \phi \|_{\HS (L^2, L^2)} \leq 2 \eta$), we obtain the existence of $v$ on $I_0$ and
\begin{align}
\| v \|_{X (I_0)} \leq 6 \eta,
\label{vbdd}
\end{align}

\noi
given that $0 < 2 \eta \leq \eta_0$ for sufficiently small absolute constant $\eta_0 > 0$ and $I_0 \subset [t_0, t_0 + \tau]$. We set $$f = \NN (v + \Psi) - \NN (v).$$ 
By the fundamental theorem of calculus, \eqref{vbdd}, and \eqref{Psi_bdd}, we have
\begin{align*}
\| f \|_{L_{I_0}^1 L_x^2} &= \bigg\| \int_0^1 \NN' (v + \al \Psi) \Psi d\al \bigg\|_{L_{I_0}^1 L_x^2} \\
&\leq C \big( \| v \|_{X (I_0)} + \| \Psi \|_{X (I_0)} \big)^{\frac{4}{d - 2}} \| \Psi \|_{X (I_0)} \\
&\leq C \big( 3 \eta_0 + C \tau^\ta \| \phi \|_{\HS (L^2, L^2)} \big)^{\frac{4}{d - 2}} \tau^\ta \| \phi \|_{\HS (L^2, L^2)} \\
&\leq C \tau^\ta \| \phi \|_{\HS (L^2, L^2)}
\end{align*}

\noi
for $\eta_0, \tau > 0$ sufficiently small. Given $\eps > 0$, we can further shrink the value of $\tau = \tau(\eps) > 0$ so that
\begin{align*}
\| f \|_{L_{I_0}^1 L_x^2} \leq \eps.
\end{align*}

\noi
Thus, by the perturbation lemma (Lemma \ref{LEM:pert}), for $0 < \eps < \eps_0$ with $\eps_0 = \eps_0 (R) > 0$ given by Lemma \ref{LEM:pert}, we obtain
\begin{align*}
\| (v - w, \dt v - \dt w ) \|_{L_{I_0}^\infty \dot{\H}_x^1} + \| v - w \|_{L_{I_0}^q L_x^r} \leq C_0 (R) \eps
\end{align*}

\noi
for some constant $C_0 (R) \geq 1$. In particular, we have
\begin{align}
\big\| \big( v(t_1) - w(t_1), \dt v (t_1) - \dt w (t_1) \big) \big\|_{\dot{\H}_x^1} \leq C_0 (R) \eps.
\label{H1_bdd}
\end{align}

If $I_0 = [t_0, t_0 + \tau]$, then we can stop the iterative argument to be performed below. If $I_0 \subsetneq [t_0, t_0 + \tau]$, then we move on to the second interval $I_1$, and we can also assume that $I_1 \subseteq [t_0, t_0 + \tau]$. By \eqref{Sw}, the Strichartz estimate (Lemma \ref{LEM:str}), and \eqref{H1_bdd}, we have
\begin{align*}
\big\| &V(t - t_1) \big( v(t_1), \dt v (t_1) \big) \big\|_{X (I_1)} \\
&\leq \big\| V(t - t_1) \big( w(t_1), \dt w (t_1) \big) \big\|_{X (I_1)} \\
&\quad + \big\| V(t - t_1) \big( w(t_1) - v(t_1), \dt w (t_1) - \dt v (t_1) \big) \big\|_{X (I_1)} \\
&\leq 2 \eta + C \cdot C_0 (R) \eps \\
&\leq 3 \eta,
\end{align*}

\noi
where we choose $\eps = \eps (R, \eta) > 0$ to be sufficiently small such that $C \cdot C_0 (R) \eps \leq \eta$. By using the above argument on $I_0$, we can obtain the existence of $v$ on $I_1$ and
\begin{align*}
\| v \|_{X (I_1)} \leq 9 \eta
\end{align*}

\noi
given $0 < 3 \eta \leq \eta_0$, and
\begin{align}
\begin{split}
\| f \|_{L_{I_1}^1 L_x^2} &\leq C \big( \| v \|_{X (I_0)} + \| \Psi \|_{X (I_0)} \big)^{\frac{4}{d - 2}} \| \Psi \|_{X (I_0)} \\
&\leq C \big( 3 \eta_0 + C \tau^\ta \| \phi \|_{\HS (L^2, L^2)} \big)^{\frac{4}{d - 2}} \tau^\ta \| \phi \|_{HS (L^2, L^2)} \\
&\leq C \tau^\ta \| \phi \|_{\HS (L^2, L^2)} \\
&\leq \eps.
\end{split}
\label{f_bdd}
\end{align}

\noi
Thus, with \eqref{H1_bdd} and \eqref{f_bdd}, by the perturbation lemma (Lemma \ref{LEM:pert}), we obtain
\begin{align*}
\| (v - w, \dt v - \dt w ) \|_{L_{I_1}^\infty \dot{\H}_x^1} + \| v - w \|_{L_{I_1}^q L_x^r} \leq C_0 (R) \cdot C_0 (R) \eps = C_0 (R)^2 \eps
\end{align*}

\noi
as long as $0 < C_0 (R) \eps < \eps_0$.

Proceeding as above iteratively, we obtain the existence of $v$ on $I_j$ for all $0 \leq j \leq J' \leq J - 1$ such that $\bigcup_{j = 0}^{J'} I_j = [t_0, t_0 + \tau]$, as long as $0 < 3\eta \leq \eta_0$ and $\eps > 0$ satisfies $C \cdot C_0 (R)^{J - 1} \eps \leq \eta \leq \frac{\eta_0}{3}$ and $C_0(R)^{J - 1} \eps < \eps_0$. For convenience, we let $\eta = \frac{\eta_0}{3}$ so that $J$ depends only on $R$ and $\eta_0$, and so the above restrictions on $\eps$ can be made possible by choosing $\eps = \eps (R, \eta_0)$ sufficiently small. This finishes the proof of Theorem \ref{THM:R}.
\end{proof}

\subsection{Global well-posedness on $\T^{d}$}
\label{SUBSEC:T}
In this subsection, we focus on global well-posedness of the defocusing energy-critical SNLW \eqref{SNLW} on the periodic setting and present the proof of Theorem \ref{THM:T}. We would like to invoke the finite speed of propagation of the wave equation to reduce our global well-posedness problem on $\T^d$ to global well-posedness on $\R^d$, which we have presented in Subsection \ref{SUBSEC:GWP_R}. However, instead of using Ito's lemma to establish an energy bound for the solution $u$ to \eqref{SNLW}, we use a Gronwall-type argument to show an a priori energy bound for the solution $v$ to the perturbed NLW \eqref{SNLWv1}.

Let us first show the a priori energy bound in a slightly more abstract setting. Consider the following perturbed equation for $v$ on $\R^d$:
\begin{align}
\begin{cases}
\dt^2 v - \Dl v + \NN (v + z) = 0 \\
(v, \dt v)|_{t = 0} = (u_0, u_1),
\end{cases}
\label{SNLMv2}
\end{align}

\noi
where $\NN (u) = |u|^{\frac{4}{d - 2}} u$ and $z$ is a space-time function satisfying some regularity assumptions to be specified below.

\begin{proposition}
\label{PROP:energy2}
Let $d \geq 3$, $(u_0, u_1) \in \H^1 (\R^d)$, and $T_0 > 0$. Let $z$ be a space-time function in the class 
\begin{align}
L^q ([0, T_0]; W^{\s, r_1} (\R^d)) \cap C\big( [0, T_0]; W^{\s - \eps, r_2} (\R^d) \big) \cap C^1( [0, T_0]; W^{\s - \eps, \infty} (\R^d)),
\label{spaces}
\end{align}
where $1 \leq q < \infty$, $2 \leq r_1 < \infty$, $2 \leq r_2 \leq \infty$, $\eps > 0$, and $\s \in \R$ satisfying
\begin{align*}
\textup{(i)}~ d = 3: ~\s > \frac 12; \qquad \textup{(ii)}~ d = 4: ~ \s > 0; \qquad \textup{(iii)}~ d \geq 5: ~ \s \geq 0.
\end{align*}

\noi
Let $v$ be the solution to the equation \eqref{SNLMv2} with $(v, \dt v)|_{t = 0} = (u_0, u_1)$ and let $T^* = T^*_\o (u_0, u_1)$ be the forward maximal time of existence. Then, there exists a constant
$$C = C\big( T_0, \| (u_0, u_1) \|_{\H^1}, \| z \|_* \big) > 0$$ such that for any stopping time $T$ with $0 < T < \min (T^*, T_0)$ almost surely, we have
\begin{align*}
\sup_{0 \leq t \leq T} E (\vec v (t)) \leq C.
\end{align*}

\noi
Here, $\| \cdot \|_*$ refers to any norm associated with the function spaces in \eqref{spaces}.
\end{proposition}

\begin{proof}
We consider three cases.

\smallskip \noi
\textbf{Case 1:} $d = 3$.

In this case, by assumption, we have $z \in C([0, T]; W^{\frac 12 + \eps', \infty} (\R^3))$ and $\dt z \in C([0, T]; W^{- \frac 12 + \eps', \infty} (\R^3))$, where $\eps' > 0$ is sufficiently small. By \eqref{Ev1} and Taylor's theorem, we have
\begin{align}
\begin{split}
\dt E (\vec v (t)) &= - \int_{\R^3} \dt v (t) \big( \NN ( v + z ) (t) - \NN (v) (t) \big) dx \\
&= - 5 \int_{\R^3} \dt v (t) \cdot |v (t)|^4 z (t) dx \\
&\quad - 10 \int_{\R^3} \dt v (t) \cdot |v (t) + \ta z (t)|^2 (v (t) + \ta z (t)) z (t)^2 dx \\
&=: I_1 + I_2
\end{split}
\label{Ev3}
\end{align}

\noi
for some $\ta \in (0, 1)$. For $I_2$, by the Cauchy-Schwarz inequality, H\"older's inequality, and Cauchy's inequality, we obtain
\begin{align}
\begin{split}
|I_2| &\les \int_{\R^3} |\dt v (t)| \big( |v (t)|^3 z (t)^2 + z (t)^5 \big) dx \\
&\leq \| \dt v (t) \|_{L^2} \big( \| z (t) \|_{L^\infty}^4 \| v (t) \|_{L^6}^6 + \| z (t) \|_{L^{10}}^{10} \big)^{\frac 12} \\
&\leq C(\| z(t) \|_{L^\infty}, \| z(t) \|_{L^{10}}) \big( 1 + E(\vec v (t)) \big).
\end{split}
\label{Ev4}
\end{align}

We now consider $I_1$. For $0 \leq t_1 \leq t_2 \leq T$, by integration by parts, H\"older's inequalities, and Young's inequalities, we obtain
\begin{align}
\begin{split}
\int_{t_1}^{t_2} I_1 dt' &= - \int_{t_1}^{t_2} \int_{\R^3} \dt \big( |v|^4 v \big) (t') z (t') dx dt' \\
&= - \int_{\R^3} |v (t_2)|^4 v (t_2) z (t_2) dx + \int_{\R^3} |v (t_1)|^4 v (t_1) z (t_1) dx \\
&\quad + \int_{t_1}^{t_2} \int_{\R^3} |v (t')|^4 v (t') \dt z (t') dx dt' \\
&\leq \dl \| v (t_2) \|_{L^6}^6 + C(\dl) \| z (t_2) \|_{L^6}^6 + \| v (t_1) \|_{L^6}^6 + \| z (t_1) \|_{L^6}^6 \\
&\quad + \int_{t_1}^{t_2} \int_{\R^3} |v (t')|^4 v (t') \dt z (t') dx dt'
\end{split}
\label{Ev5}
\end{align}

\noi
for some $0 < \dl < 1$. By duality, H\"older's inequality, Lemma \ref{LEM:chain}, and Lemma \ref{LEM:Gag}, we have
\begin{align}
\begin{split}
\int_{t_1}^{t_2} &\int_{\R^3} |v (t')|^4 v (t') \dt z (t') dx dt' \\
&= \int_{t_1}^{t_2} \int_{\R^3} \jb{\nb}^{\frac 12 - \eps'} \big( |v|^4 v \big) (t') \jb{\nb}^{- \frac 12 + \eps'} (\dt z) (t') dx dt' \\
&\les \int_{t_1}^{t_2} \big\| \jb{\nb}^{\frac 12 - \eps'} \big( |v|^4 v \big) (t') \big\|_{L^{\frac{3}{3 - \eps'}}} \big\| \jb{\nb}^{-\frac 12 + \eps'} (\dt z) (t') \big\|_{L^\infty} dt' \\
&\les \| \dt z \|_{L^\infty_T W_x^{-\frac 12 + \eps', \infty}} \int_{t_1}^{t_2} \| v(t') \|_{L^6}^4 \big\| \jb{\nb}^{\frac 12 - \eps'} v (t') \big\|_{L^{\frac{3}{1 - \eps'}}} dt' \\
&\leq \| \dt z \|_{L^\infty_T W_x^{-\frac 12 + \eps', \infty}} \int_{t_1}^{t_2} E(\vec v (t'))^{\frac 23} \| \jb{\nb} v (t') \|_{L^2}^{\frac 12} \| v (t') \|_{L^6}^{\frac 12} dt' \\
&\leq \| \dt z \|_{L^\infty_T W_x^{-\frac 12 + \eps', \infty}} \bigg( 1 + \int_{t_1}^{t_2} E(\vec v (t')) dt' \bigg).
\end{split}
\label{Ev6}
\end{align}

Combining \eqref{Ev3}, \eqref{Ev4}, \eqref{Ev5}, and \eqref{Ev6}, we obtain
\begin{align*}
E(\vec v (t_2)) &\leq C (\| z \|_*) \int_{t_1}^{t_2} E(\vec v (t')) dt' + C ( \| z \|_*, \| v(t_1) \|_{L^6} )
\end{align*}

\noi
for any $0 \leq t_1 \leq t_2 \leq T$. By Gronwall's inequality and Sobolev's embedding, we get
\begin{align*}
E (\vec v (t)) \leq C(\| (u_0, u_1) \|_{\H^1}) e^{C(\| z \|_*, \| (u_0, u_1) \|_{\H^1}) t}
\end{align*}

\noi
for any $0 < t \leq T$, which implies the desired energy bound.

\smallskip \noi
\textbf{Case 2:} $d = 4$.

In this case, by assumption, we have $z \in C([0, T]; L^{r_2} (\R^4))$ for any $2 \leq r_2 \leq \infty$. Thus, by~\eqref{SNLMv2}, we have
\begin{align}
\begin{split}
\dt E (\vec v (t)) &= \int_{\R^d} \dt v (t) \big( \dt^2 v (t) - \Dl v (t) + v (t)^3 \big) dx \\
&= - \int_{\R^4} \dt v (t) \big( \NN ( v + z ) (t) - \NN (v) (t) \big) dx.
\end{split}
\label{Ev1}
\end{align}

\noi
By the fundamental theorem of calculus, we have
\begin{align}
\big| \NN ( v + z ) - \NN (v) \big| = \bigg| \int_0^1 \NN' (v + \al z) z d \al \bigg| \les |z| |v|^{2} + |z|^{3}.
\label{Ev2}
\end{align}

\noi
Thus, by \eqref{Ev1}, \eqref{Ev2}, H\"older's inequality, and the Cauchy-Schwarz inequality, we obtain
\begin{align*}
\dt E (\vec v (t)) &\les \| z (t) \|_{L^\infty} \int_{\R^d} |\dt v (t)| |v (t)|^{2} dx + \int_{\R^d} |\dt v (t)| |z (t)|^{3} dx \\
&\leq \| z (t) \|_{L^\infty} \| \dt v (t) \|_{L^2} \| v (t) \|_{L^{4}}^{2} + \| z (t) \|_{L^6}^{3} \| \dt v (t) \|_{L^2} \\
&\leq C ( \| z (t) \|_{L^\infty}, \| z(t) \|_{L^{6}} ) \big( 1 + E (\vec v (t)) \big).
\end{align*}

\noi
By Gronwall's inequality, we get
\begin{align*}
E (\vec v (t)) \leq C(\| (u_0, u_1) \|_{\H^1}) e^{C(\| z \|_*) t}
\end{align*}

\noi
for any $0 < t \leq T$, which implies the desired energy bound.

\smallskip \noi
\textbf{Case 3:} $d \geq 5$.

In this case, by assumption, we have $z \in L^q ([0, T]; L^{r_1} (\R^d))$ for $1 \leq q < \infty$ and $2 \leq r_1 < \infty$. Proceeding as in Case 2 using the fundamental theorem of calculus, we obtain
\begin{align}
\begin{split}
\dt E (\vec v (t)) &\les \int_{\R^d} |\dt v (t)| |v (t)|^{\frac{4}{d - 2}} |z (t)| dx + \int_{\R^d} |\dt v (t)| |z (t)|^{\frac{d + 2}{d - 2}} dx \\
&=: I_3 + I_4.
\end{split}
\label{E3-1}
\end{align}

\noi
Let $0 \leq t_1 \leq t_2 \leq T$. For $I_4$, by the Cauchy-Schwarz inequalities, we obtain
\begin{align}
\int_{t_1}^{t_2} I_4 dt' \leq \| z \|_{L^2_T L_x^{\frac{2d + 4}{d - 2}}}^{\frac{d + 2}{d - 2}} \bigg( 1 + \int_{t_1}^{t_2} E (\vec v (t')) dt' \bigg).
\label{E3-2}
\end{align}

\noi
For $I_1$, by H\"older's inequalities, we obtain
\begin{align}
\begin{split}
\int_{t_1}^{t_2} I_3 dt' &\leq \int_{t_1}^{t_2} \| \dt v (t') \|_{L^2} \| v (t') \|_{L^{\frac{2d}{d - 2}}}^{\frac{4}{d - 2}} \| z (t') \|_{L^{\frac{2d}{d - 4}}} dt' \\
&\les \int_{t_1}^{t_2} E (\vec v (t'))^{\frac{d + 4}{2d}} \| z (t') \|_{L^{\frac{2d}{d - 4}}} dt' \\
&\les \| z \|_{L_T^{\frac{2d}{d - 4}} L_x^{\frac{2d}{d - 4}}} \bigg( 1 + \int_{t_1}^{t_2} E (\vec v (t')) dt' \bigg).
\end{split}
\label{E3-3}
\end{align}

\noi
Thus, combining \eqref{E3-1}, \eqref{E3-2}, \eqref{E3-3} and using Gronwall's inequality, we obtain
\begin{align*}
E (\vec v (t)) \leq C(\| (u_0, u_1) \|_{\H^1}) e^{C(\| z \|_*, \| (u_0, u_1) \|_{\H^1}) t}
\end{align*}

\noi
for any $0 < t \leq T$, which implies the desired energy bound.
\end{proof}

\begin{remark} \label{RMK:dtE} \rm
(i) To make the computations in the above proof more rigorous, we need to work with smooth solutions $(v_N, \dt v_N)$ to \eqref{SNLMv2} with truncated initial data $(P_{\leq N} u_0, P_{\leq N} u_1)$ and truncated perturbation term $P_{\leq N} z$, where $P_{\leq N}$ denotes the sharp frequency projection onto $\{ |n| \leq N \}$. Using similar arguments as in Lemma \ref{LEM:uN}, we can show that $(v_N, \dt v_N)$ converges to $(v, \dt v)$ in $C([0, T]; \dot{\H}^1 (\R^d))$. After establishing an upper bound for $E (v_N , \dt v_N )$ that is independent of $N$, we can take $N \to \infty$ to obtain the desired energy bound for $E (\vec v)$. 

\smallskip \noi
\text{(ii)} In the case when $d = 4$ and $\s = 0$, the argument in Case 3 in the above proof breaks down. Due to the lack of an $L^\infty$-bound in the endpoint case $\s = 0$, we have to rely on the fact that $\frac{8}{d - 2} < \frac{2d}{d - 2}$ when dealing with $I_3$ term in the above proof. However, when $d = 4$, the inequality $\frac{8}{d - 2} < \frac{2d}{d - 2}$ does not hold and we do not know how to overcome this issue at this point. It seems that a more intricate analysis is needed to establish an energy bound for the endpoint case $\s = 0$ for $d = 4$.
\end{remark}

We are now ready to show Theorem \ref{THM:T} by reducing it to the $\R^d$ case. To achieve this, we adjust the use of the finite speed of propagation as in \cite[Appendix A]{OP17} to our stochastic setting.
\begin{proof}[Proof of Theorem \ref{THM:T}]
Let $T > 0$ be a target time and $N \in \N$. Given initial data $(u_0, u_1) \in \H^1 (\T^d)$, we define
\begin{align*}
u_{j, T} (x) &:= \eta_T (x) \sum_{n \in \Z^d} \ft{u_j} (n) e^{i n \cdot x}, \quad x \in \R^d \\
u_{j, N, T} (x) &:= \eta_T (x) \sum_{|n| \leq N} \ft{u_j} (n) e^{i n \cdot x}, \quad x \in \R^d
\end{align*}

\noi
for $j = 0, 1$, where $\eta_T$ is a smooth cutoff function on $[-2 \pi T, 2 \pi T]$ as defined in \eqref{etaR}. 



We consider the following equation for $\mathbf{u}_{N}$ on $\R^d$:
\begin{align}
\begin{cases}
\dt^2 \mathbf{u}_{N, T} - \Delta \mathbf{u}_{N, T} + \NN (\mathbf{u}_{N, T}) = \eta_T P_{\leq N} (\phi \xi) \\
(\mathbf{u}_{N, T}, \dt \mathbf{u}_{N, T})|_{t = 0} = (u_{0, N, T}, u_{1, N, T}),
\end{cases}
\label{SNLW_NT}
\end{align}

\noi
where $P_{\leq N}$ denotes the sharp frequency cutoff onto $\{|n| \leq N\}$ and $\NN (u) = |u|^{\frac{4}{d - 2}} u$.
Since the initial data and the noise are smooth, there exists a unique (smooth) global solution $\mathbf{u}_{N, T}$ to \eqref{SNLW_NT} on $\R^d$. By the finite speed of propagation, we have that $u_N := \mathbf{u}_{N}|_{[0, T] \times \T^d}$ is a solution to the following equation on $\T^d$:
\begin{align*}
\begin{cases}
\dt^2 u_N - \Delta u_N + \NN (u_N) = P_{\leq N} (\phi \xi) \\
(u_N, \dt u_N)|_{t = 0} = (P_{\leq N} u_0, P_{\leq N} u_1).
\end{cases}
\end{align*}

We recall the definition of $\mathbf{\Phi}_{T}$ in \eqref{defPhiR} and define $\mathbf{\Phi}_{N, T}$ as $\mathbf{\Phi}_{T}$ with the summation restricted to $|n| \leq N$. By Lemma \ref{LEM:Phi}, it is not hard to see that $\mathbf{\Phi}_{N, T}$ converges to $\mathbf{\Phi}_T$ in $C([0, T]; W^{s + 1 - \eps, \infty} (\R^d))$ given $\phi \in \HS (L^2, H^s)$. Note that $\mathbf{\Phi}_{N, T}$ and $\mathbf{\Phi}_T$ satisfy
\begin{align*}
\dt^2 \mathbf{\Phi}_{N, T} - \Delta \mathbf{\Phi}_{N, T} = \eta_T P_{\leq N} (\phi \xi) \quad \text{and} \quad
\dt^2 \mathbf{\Phi}_T - \Delta \mathbf{\Phi}_T = \eta_T (\phi \xi),
\end{align*}

\noi
respectively, both with zero initial data.
We also recall that $\Phi$ is the stochastic convolution as defined in \eqref{defPhi}
and we let $\Phi_{N}$ be the truncation of $\Phi$ onto frequencies $\{|n| \leq N\}$. By the finite speed of propagation for the linear solutions, we have
\begin{align*}
\Phi_N = \mathbf{\Phi}_{N, T}|_{[0, T] \times \T^d} \quad \text{and} \quad \Phi = \mathbf{\Phi}_T |_{[0, T] \times \T^d}.
\end{align*}

We now define $\mathbf{v}_{N, T} = \mathbf{u}_{N, T} - \mathbf{\Phi}_{N, T}$, which is the smooth global solution to the following perturbed NLW on $\R^d$:
\begin{align*}
\begin{cases}
\dt^2 \mathbf{v}_{N, T} - \Delta \mathbf{v}_{N, T} + \NN (\mathbf{v}_{N, T} + \mathbf{\Phi}_{N, T}) = 0 \\
(\mathbf{v}_{N, T}, \dt \mathbf{v}_{N, T})|_{t = 0} = (u_{0, N, T}, u_{1, N, T}).
\end{cases}
\end{align*}

\noi
In particular, $v_N := \mathbf{v}_{N, T}|_{[0, T] \times \T^d}$ is the smooth global solution to the following perturbed NLW on $\T^d$:
\begin{align*}
\begin{cases}
\dt^2 v_N - \Delta v_N + \NN (v_N + \Phi_N) = 0 \\
(v_N, \dt v_N)|_{t = 0} = (P_{\leq N} u_0, P_{\leq N} u_1).
\end{cases}
\end{align*}

We also consider the following equation for $\mathbf{v}_T$ on $\R^d$:
\begin{align}
\begin{cases}
\dt^2 \mathbf{v}_T - \Delta \mathbf{v}_T + \NN (\mathbf{v}_T + \mathbf{\Phi}_T) = 0 \\
(\mathbf{v}_T, \dt \mathbf{v}_T)|_{t = 0} = (u_{0, T}, u_{1, T}).
\end{cases}
\label{NLWvT}
\end{align}

\noi
It is easy to see that $(u_{0, T}, u_{1, T}) \in \H^1 (\R^d)$ given $(u_0, u_1) \in \H^1 (\T^d)$. In addition, given the range of $s$ provided in the statement of Theorem \ref{THM:T}, we obtain by Lemma \ref{LEM:Phi} that
\begin{align*}
\mathbf{\Phi}_T \in L^q ([0, T]; W^{\s, r_1} (\R^d)) \cap C\big( [0, T]; W^{\s - \eps, r_2} (\R^d) \big) \cap C^1( [0, T]; W^{\s - \eps, \infty} (\R^d)),
\end{align*}

\noi
where $1 \leq q < \infty$, $2 \leq r_1 < \infty$, $2 \leq r_2 \leq \infty$, $\eps > 0$, and $\s \in \R$ satisfying
\begin{align*}
\textup{(i)}~ d = 3: ~\s > \frac 12; \qquad \textup{(ii)}~ d = 4: ~ \s > 0; \qquad \textup{(iii)}~ d \geq 5: ~ \s \geq 0.
\end{align*}

\noi
Thus, by Proposition \ref{PROP:energy2}, we have an a priori energy bound for $\vec{\mathbf{v}}_T$, so that repeating the proof of Theorem \ref{THM:R} yields global well-posedness of the equation \eqref{NLWvT} (note that we do not need $H^1(\R^d)$-regularity of $\mathbf{\Phi}_T$).

By using a similar argument as in Lemma \ref{LEM:uN}, we can show that
\begin{align*}
\| \vec{\mathbf{v}}_{T} - \vec{\mathbf{v}}_{N, T} \|_{L_T^\infty \dot{\H}^1_x (\R^d)} \too 0
\end{align*}

\noi
and
\begin{align*}
\| \mathbf{v}_T - \mathbf{v}_{N, T} \|_{X([0, T] \times \R^d)} \too 0
\end{align*}

\noi
as $N \to \infty$, where we recall that the $X(I \times \R^d)$-norm is as defined in \eqref{defX}. See also \cite[Appendix A]{OP17} for the precise steps for obtaining the convergence of $\vec{\mathbf{v}}_{N, T}$ to $\vec{\mathbf{v}}_{T}$. This in particular implies that $v_N = \mathbf{v}_{N, T}|_{[0, T] \times \T^d}$ converges to $v := \mathbf{v}_T |_{[0, T] \times \T^d}$ in $C([0, T]; \dot{H}^1 (\T^d) ) \cap X ([0, T] \times \T^d)$ and that $\dt v_N$ converges to $\dt v$ in $C([0, T]; L^2 (\T^d))$. Since $\mathbf{v}_T$ satisfies \eqref{NLWvT} and $\Phi = \mathbf{\Phi}_T |_{[0, T] \times \T^d}$, we have that $v$ satisfies the perturbed NLW on $\T^d$:
\begin{align*}
\begin{cases}
\dt^2 v - \Delta v + \NN (v + \Phi) = 0 \\
(v, \dt v)|_{t = 0} = (u_0, u_1),
\end{cases}
\end{align*}

\noi
and that $v$ satisfies the following Duhamel formulation:
\begin{align}
v(t) = V_{\T^d} (t) (u_0, u_1) - \int_0^t S_{\T^d} (t - t') \big( \NN (v + \Phi) (t') \big) dt',
\label{Duhv}
\end{align}

\noi
where $V_{\T^d}$ and $S_{\T^d}$ has the same forms as $V$ and $S$ in \eqref{defS} but on the periodic domain $\T^d$. From \eqref{Duhv}, we can easily see that $v \in C([0, T]; L^2 (\T^d))$. Thus, we deduce that $u := \Phi + v$ is a solution to \eqref{SNLW} on $\T^d$ in the class
\begin{align*}
\Phi + C([0, T]; \H^1 (\T^d)) \subset C([0, T]; \H^{s + 1 - \eps} (\T^d)),
\end{align*}

\noi
where the inclusion follows from Lemma \ref{LEM:Phi}.

\end{proof}

\begin{ackno}\rm
The authors would like to thank Professor Tadahiro Oh for suggesting this problem and for his support and advice throughout the entire project. E.B. was supported by Unit\'e de Math\'ematiques Pures et Appliqu\'ees UMR 5669 ENS de Lyon / CNRS. 
G.L. was supported by the EPSRC New Investigator Award (grant no. EP/S033157/1). 
and the European Research Council (grant no. 864138 ``SingStochDispDyn'').
R.L. was supported by the European Research Council (grant no. 864138 ``SingStochDispDyn'').
\end{ackno}

\end{document}